\documentclass[12pt]{article}
\usepackage[latin1]{inputenc}

\usepackage{amsmath}
\usepackage{amsfonts}
\usepackage{amssymb}
\usepackage{graphicx, dsfont}
\usepackage{amsthm}
\usepackage{setspace}
\usepackage{booktabs}%\toprule macro etc
\usepackage{tabularx}%tables
\usepackage{color} %colours!
\usepackage{enumerate}
\usepackage{bbm}
\usepackage{xcolor}
\usepackage[title]{appendix}
\usepackage[shortlabels]{enumitem} 

\usepackage{thmtools}
\usepackage{thm-restate}
\usepackage{hyperref}
\usepackage{cleveref}

\usepackage{tikz}

\usepackage[textsize=footnotesize,color=green!40]{todonotes}

\newcounter{i}
\newtheorem{theorem}{Theorem} 
\newtheorem{lemma}[theorem]{Lemma}

\newtheorem{conjecture}[theorem]{Conjecture}
\newtheorem{corollary}[theorem]{Corollary}
\newtheorem{claim}[theorem]{Claim}

\theoremstyle{definition}
\newtheorem{definition}[theorem]{Definition}

\newenvironment{proofclaim}[1][]%
{\noindent \emph{Proof.} {}{#1}{}}{\hfill
	$\Diamond$\vspace{1em}}

\theoremstyle{plain} % just in case the style had changed
\newcommand{\thistheoremname}{}
\newtheorem{genericthm}[section]{\thistheoremname}

\newcommand{\es}{\emptyset}
\newcommand{\eps}{\varepsilon}
\newcommand{\sm}{\setminus}
\renewcommand{\subset}{\subseteq}

\newcommand{\NATS}{\mathbb{N}}

\newcommand{\om}[1]{|#1|_\mu}
\newcommand{\Om}{\Omega}

\newcommand{\M}{\mathcal{M}}
\newcommand{\MP}{\mathcal{MP}}
\DeclareMathOperator{\K}{Keep}
\DeclareMathOperator{\Eq}{Eq}

\renewcommand{\Pr}{\mathbf{Pr}}
\newcommand{\Exp}{\mathbf{E}}

\newcommand{\Prob}[1]{\ensuremath{%
		\mathbf{Pr} \left[#1\right]
}}

\newcommand{\Expect}[1]{\ensuremath{%
		\mathbf{E}\left[#1\right]
}}

\DeclareMathOperator{\Med}{Med}

\renewcommand{\L}{\ell}
\newcommand{\N}{n}
\newcommand{\e}{\mathrm{e}}

\title{Edge-colouring graphs with local list sizes}
\author{
	Marthe Bonamy
	\thanks{CNRS, LaBRI, Universit\'{e} de Bordeaux, France {\tt marthe.bonamy@u-bordeaux.fr}.}
	\and
	Michelle Delcourt
	\thanks{Department of Mathematics, Toronto Metropolitan University,
		Toronto, Ontario M5B 2K3, Canada {\tt mdelcourt@torontomu.ca}. Research supported by supported by NSERC under Discovery Grant No. 2019-04269.}
	\and
	Richard Lang
	\thanks{Fachbereich Mathematik, Universit\"{a}t Hamburg, 20146 Hamburg, Germany {\tt richard.lang@uni-hamburg.de}.}
	\and
	Luke Postle
	\thanks{Combinatorics and Optimization Department,
		University of Waterloo, Waterloo, Ontario N2L 3G1, Canada {\tt lpostle@uwaterloo.ca}. Partially supported by NSERC
		under Discovery Grant No. 2019-04304.}}
\date{\today}

\begin{document}
	\font\smallrm=cmr8
	\maketitle
	
	\begin{abstract}
		The famous List Colouring Conjecture from the 1970s states that for every graph $G$ the chromatic index of $G$ is equal to its list chromatic index.  In 1996 in a seminal paper, Kahn proved that the List Colouring Conjecture holds asymptotically.  Our main result is a local generalization of Kahn's theorem.  
		More precisely, we show that, for a graph $G$ with sufficiently large maximum degree $\Delta$ and minimum degree $\delta \geq \ln^{25} \Delta$, the following holds: for every assignment $L$ of lists of colours to the edges of $G$, such that $|L(e)| \geq (1+o(1)) \cdot  \max\left\{\deg(u),\deg(v)\right\}$ for each edge $e=uv$, there is an $L$-edge-colouring of $G$. Furthermore, Kahn showed that the List Colouring Conjecture holds asymptotically for linear, $k$-uniform hypergraphs, and recently Molloy generalized Kahn's original result to correspondence colouring as well as its hypergraph generalization.  We prove local versions of all of these generalizations by showing a weighted version that simultaneously implies all of our results.
	\end{abstract}

	\section{Introduction} 
	A \emph{$k$-edge-colouring} of a graph $G$ is an assignment of a set of $k$ colours to the edges of $G$ so that no two incident edges receive the same colour.  The \emph{chromatic index} of $G$, denoted $\chi'(G)$, is the minimum integer $k$ such that $G$ has a $k$-edge-colouring.  A natural generalization of this concept is list edge-colouring introduced independently by Vizing~\cite{Vizing_engl} as well as by Erd\H{o}s, Rubin, and Taylor~\cite{ERT79}.  An \emph{$L$-edge-colouring} of a graph $G$ is an edge-colouring in which each edge $e$ receives a colour from a prescribed list $L(e)$ of permissible colours. A classic problem in list edge-colouring is to determine lower bounds for $|L(e)|$ that guarantee that there is an $L$-edge-colouring for all lists $L(e)$ satisfying these conditions.
	Much of the research in this area has focused on global bounds, where all lists are bounded from below by the same parameter.
	More precisely, the \emph{list chromatic index}, denoted by $\chi'_{\ell}(G)$, is defined as the least $k$ such that there is an $L$-edge-colouring whenever $|L(e)| \geq k$ for all edges $e$.  By considering lists of the form $L(e) = \left\{1, 2, \ldots ,\chi'(G)\right\}$ for all edges $e \in E(G)$, we see that these parameters are related by the inequality $\chi'_{\ell}(G) \geq \chi'(G)$.  
	
	The famous \textit{List (Edge) Colouring Conjecture} suggests that something stronger is true.
	\begin{conjecture}\label{con:list-colouring-conjecture}
		If $G$ is a graph, then $\chi'_{\ell}(G) = \chi'(G)$.
	\end{conjecture}
	\noindent
	This conjecture has been suggested independently by a number of researchers during the 1970s and 1980s; for more details on the history of the List Colouring Conjecture, see Problem 12.20 in Jensen and Toft~\cite{JT95}.  The List Colouring Conjecture has been confirmed for several classes of graphs including $d$-regular, $d$-edge-colourable planar graphs by Ellingham and Goddyn~\cite{EG96}, complete graphs of odd order by H\"aggkvist and Janssen~\cite{HJ97}, and complete graphs of prime degree by Schauz~\cite{Sch14}.  One of the best known results in this vein is the following theorem by Galvin~\cite{Gal95}, proving Dinitz's conjecture from the 1950s. 
	
	\begin{theorem}[Galvin]\label{Galvin}
		If $G$ is a bipartite graph, then 
		$\chi'_{\ell}(G) = \chi'(G).$
	\end{theorem}
	
	\noindent
	In 1996 Kahn~\cite{JK} showed that Conjecture~\ref{con:list-colouring-conjecture} holds asymptotically as follows.  Let $\Delta(G)$ denote the maximum degree of a graph $G$, and recall that Vizing~\cite{Viz} proved that $\Delta(G) \leq \chi'(G) \leq \Delta(G) + 1$ for all graphs $G$.  
	
	\begin{theorem}[Kahn]\label{thm:kahn}
		For every $\eps>0$, 
		if $G$ is a graph with sufficiently large maximum degree $\Delta(G)$, then $\chi'_{\ell}(G) \leq (1+\eps)\Delta(G)$.
	\end{theorem}
	
	In this paper, we study list edge-colourings under local conditions, where $|L(e)|$ is lower bounded by a function that takes into account the  local structure around $e$.  For vertex colourings, such local notions have been recently studied in terms of local clique sizes by Bonamy, Kelly, Nelson, and Postle~\cite{BKNL18} and for triangle-free graphs by Davies, de~Joannis~de Verclos, Kang, and Pirot~\cite{DJK20}.  
	For edge-colourings, local generalizations appeared as early as in the work of Erd\H{o}s, Rubin, and Taylor~\cite{ERT79}.  
	Borodin, Kostochka, and Woodall showed a local generalization of Galvin's theorem~\cite{BKW97}.
	\begin{theorem}[Borodin, Kostochka, and Woodall]\label{BKW}
		If $G$ is a bipartite graph and $L$ is a list assignment of $E(G)$ such that for every edge $e=uv$
		$$|L(e)|\ge \max \{\deg(u),\deg(v)\},$$
		then there is an $L$-edge-colouring of $G$.
	\end{theorem}
	
	Moreover, they proved a local result for general graphs, namely that if for every edge $e=uv$,
	$|L(e)| \geq \max \{\deg(u),\deg(v)\}+\left\lfloor \frac{1}{2}\min\{\deg(u),\deg(v)\}\right\rfloor$, then $G$ is $L$-edge-colourable.  We remark that Galvin's theorem and Borodin, Kostochka, and Woodall's results also hold for multigraphs.

	Our first main result is a local analogue of Kahn's theorem under the condition that the maximum degree is polylogarithmic in terms of the minimum degree.
	
	\begin{theorem}\label{thm:main-simple}
		For every $\eps >0$ the following holds: if $G$ is a graph with sufficiently large maximum degree $\Delta(G)$, minimum degree $\delta(G) \geq \ln^{25} \Delta(G)$, and  $L$ is a list assignment for $E(G)$ such that for every edge $e=uv$
		$$|L(e)|\ge (1+\eps)  \max \{\deg(u),\deg(v)\},$$
		then there is an $L$-edge-colouring of $G$. 
	\end{theorem}

	{We remark that as demonstrated by Amini, Esperet, and Van Den Heuvel~\cite{AEV13} a weaker version of Theorem~\ref{thm:main-simple} with linear dependency between the maximum and minimum degree follows implicitly from Kahn's proof.}  
	
	\subsection{Weighted local generalizations for list colouring} 
	
	The key idea to prove Theorem~\ref{thm:main-simple} {is to show an analogous result about weighted colours}.  We assume without loss of generality that all colours $c \in \NATS$.  
	An assignment $(L,\mu)$ of \emph{weighted} lists of colours to the edges of $G$ consists of lists of colours $L(e)$ and weight functions $\mu(e)\colon L(e) \to (0,1]$.
	For convenience of notation, we let $\mu(e,c)$ denote $\mu(e)(c)$.  We write $\om{L(e)}:=\sum_{c \in L(e)} \mu(e,c)$ and $\om{A} :=\sum_{(e,c) \in A} \mu(e,c)$ for a set $A \subset E \times \NATS$. 
	
	\begin{theorem}\label{thm:maingraph}
		For every $\eps >0$, there exists ${\delta_0=\delta_0(\eps)}$ such that the following holds for all $\delta \geq \delta_0$:
		let $G =(V,E)$ be a graph with weighted lists of colours $(L,\mu)$. 
		If for every edge $e \in E$, we have
		\begin{enumerate}[\upshape (a)]
			\item $\mu(e)\colon L(e) \to \left[ \exp\left(-\delta^{1/25}\right) ,\delta^{-1}\right]$, and
			\item $\om{L(e)} \geq 1+\eps$, and
			\item for every vertex $v \in V$ and colour $c$, $\sum_{e\sim v} \mu(e,c) \leq 1$,
		\end{enumerate}
		then there is an $L$-edge-colouring of $G$.
	\end{theorem}
	
	We see that Theorem~\ref{thm:main-simple} follows as a corollary.
	\begin{proof}[Proof of Theorem~\ref{thm:main-simple}]
		We obtain $\delta_0$ from Theorem~\ref{thm:maingraph} with input $\eps$. 
		Let $G$, $\delta=\delta(G) \geq \delta_0$, $\Delta=\Delta (G)$, and $L(e)$ be as in the statement of Theorem~\ref{thm:main-simple}.
		For every edge $e =uv $ and colour $c \in L(e)$, we set $\mu(e,c)=\frac{1}{\max \{\deg(u),\deg(v)\}}$. 
		Thus, for each edge $e=uv$ and colour $c \in L(e)$, we have $$\exp\left(-\delta^{1/25}\right) \leq  \Delta^{-1} \leq \mu(e,c) \leq \delta^{-1},$$ 
		$$\om{L(e)} \geq 	\tfrac{|L(e)|}{\max \{\deg(u),\deg(v)\}}  \geq  1+\eps ,$$
		and for every vertex $v$ and colour $c$,$$\sum_{e \sim v} \mu(e,c) = \sum_{e=wv \in E(G)} \frac{1}{\max \{\deg(w),\deg(v)\}} \leq \sum_{e \sim v} \frac{1}{ \deg(v)}\leq 1.$$
		Therefore by Theorem~\ref{thm:maingraph} there exists an $L$-edge-colouring of $G$.
	\end{proof}
	
	A \emph{$k$-uniform hypergraph} is a hypergraph with all edges containing exactly $k$ vertices.  A hypergraph is said to be \emph{linear} if all pairs of distinct edges intersect in at most one vertex. Kahn~\cite{JK} more generally showed that the List Colouring Conjecture holds asymptotically for linear, $k$-uniform hypergraphs. In fact our main technical theorem, Theorem~\ref{thm:main}, shows that Theorem~\ref{thm:maingraph} holds more generally for linear, $k$-uniform hypergraphs.

	\subsection{Correspondence colouring}

	\newcommand{\Lc}{\mathcal{L}}
	\newcommand{\Nc}{\mathcal{N}}
	Recently Molloy and Postle~\cite{MP22} expanded Kahn's results to correspondence colouring, a generalization of list colouring. We are able to prove Theorem~\ref{thm:main-simple} and Theorem~\ref{thm:maingraph} simultaneously by showing Theorem~\ref{thm:main}.  This is our main technical theorem and is a local version of Molloy's result. To state this formally, we need a few further definitions.

	\newcommand{\LL}{\mathcal{L}}
	Let $G=(V,E)$ be a linear, $k$-uniform hypergraph.  An \emph{edge correspondence} $\sigma$ of $G$ consists of integer permutations $\sigma_{e,f}=\sigma_{f,e}^{-1}$ for all edges $e \sim f$.
	For edges $e \sim f$ and colours $c,c' \in \NATS$, we say that $(e,c)$ \emph{blocks} $(f,c')$ if $\sigma_{e,f}(c) = c'$.
	We define an \emph{$(L,\sigma)$-colouring} to be a function $\gamma \colon E \to \NATS$ such that
	\begin{itemize}
		\item $\gamma(e) \in L(e)$ for every $e \in E$, and
		\item $(e,\gamma(e))$ does not block $(f,\gamma(f))$ for all edges $e \sim f$.
	\end{itemize}
	Finally, we define the \emph{colour neighbours}, denoted by $N_{G,L,\sigma}(e,v,c)$, as the set containing all pairs $(f,c') \in E \times \NATS$ such that $f$ is incident to $v$, $f \neq e$, $c' \in L(f)$ and $(f,c')$ blocks $(e,c)$.
	
	We are now ready to state our main technical theorem:
	
	\begin{theorem}\label{thm:main}
		For every $k \in \NATS$ and $\eps >0$, there exists $\delta_0={\delta_0}(\eps,k) \in \NATS$ such that the following holds for all $\delta \geq \delta_0$: 
		Let $G =(V,E)$ be a $k$-uniform linear hypergraph with an edge correspondence $\sigma$ and weighted lists of colours $(L,\mu)$. 
		If for every edge $e \in E$, we have
		\begin{enumerate}[\upshape (a)]
			\item \label{itm:con-main-thm-mu} $\mu(e)\colon L(e) \to \left[ \exp\left(-\delta^{1/25}\right) ,\delta^{-1}\right]$, and
			\item \label{itm:con-main-thm-L}  $\om{L(e)} \geq 1+\eps$, and
			\item \label{itm:con-main-thm-N} $\om{N_{G,L,\sigma}(e,v,c)} \leq 1$ for every vertex $v \in e$ and colour $c \in L(e)$,
		\end{enumerate}
		then there is an $(L,\sigma)$-colouring of $G$.
	\end{theorem}

	\subsection{The matching polytope and list edge-colourings}\label{sec:matching-polotype}
	We now discuss one further application of Theorem~\ref{thm:maingraph}. Let $G$ be a graph with $m$ edges and $\M(G)$ denote the set of all matchings of $G$.  For each matching $M \in \M(G)$, we assign an $m$-dimensional characteristic vector $\mathds{1}_M = (x_e)_{e \in E(G)}$, where $x_e=1$ if $e$ is an edge in $M$ and $x_e=0$ otherwise.  The \emph{matching polytope of $G$}, denoted $\MP(G)$, is the convex hull of all vectors of the form $\mathds{1}_M$ where $M\in \M(G)$. Implicitly, Kahn \cite{K00} proved the following result in this setting (appearing as Theorem 2.11 in Amini, Esperet, and Van Den Heuvel~\cite{AEV13}).
	\begin{theorem}[Kahn]\label{thm:mp}
		For every $\delta, \nu \in \mathbb{R}$, with $0 < \delta < 1$ and $\nu > 0$, there exists a constant $\Delta_0({\delta, \nu})$ such that for all $\Delta \geq {\Delta_0({\delta, \nu})}$, if $G$ is a graph with maximum degree at most $\Delta$ and $L$ is a list assignment of $E(G)$ such that
		\begin{enumerate}[\upshape (i)]
			\item for all $e \in E(G)$, $|L(e)| \geq \nu \cdot \Delta$, and 
			\item the vector $\vec{x} = \left(\frac{1}{|L(e)|}\colon e \in E(G)\right)$ is an element of $(1-\delta)\MP(G)$,
		\end{enumerate}
		then $G$ is $L$-edge-colourable.
	\end{theorem}
	Informally speaking, this is a local version where the hypothesis is that the list sizes ``reside'' in the interior of the matching polytope.  We remark that here there is a linear dependency between the maximum and minimum degree. Our main result Theorem~\ref{thm:main} implies this and more.  In our version the minimum degree is polylogarithmic in terms of the maximum degree:
	\begin{theorem}\label{thm:mp2}
		Theorem~\ref{thm:mp} holds with condition (i) replaced by:\\ for all $e \in E(G)$, $|L(e)| \geq \ln^{25}(\Delta)$.
	\end{theorem} 
	
	In 1965, Edmonds~\cite{E65} showed the following characterization of the matching polytope.  
	A vector $\vec{x} = (x_e)$ is in $\MP(G)$ if and only if all of the following hold:
	$x_e \geq 0$ for all $x_e$,  $\sum_{e\sim v} x_e \leq 1$ for every $v \in V(G)$, and for every $W \subseteq V(G)$ with $|W|\geq 3$ and $|W|$ odd, $\sum_{e \in E(W)} x_e \leq \frac{1}{2}(|W|-1).$  Clearly, $\vec{x}$ with $x_e = \mu(e,c) = \frac{1}{(1-\delta)|L(e)|}$ satisfies the first two conditions of Edmonds' characterization if and only if $\mu(e,c) = \frac{1}{(1-\delta){|L(e)|}}$ satisfies conditions $(b)$ and $(c)$ of Theorem~\ref{thm:maingraph} for $\varepsilon = \frac{\delta}{1 - \delta}$.
	Thus, Theorem~\ref{thm:mp2} follows directly from Theorem~\ref{thm:maingraph} and holds more generally, as we do not need $\vec{x}$ to satisfy the cut condition in Edmonds' characterization.

	\medskip
	The rest of the paper is organized as follows.
	In the next section, we formulate our two main technical lemmas (Lemma~\ref{lem:edge-finisher} and~\ref{lem:nibble}) and use them to prove Theorem~\ref{thm:main}.
	In Section~\ref{sec:tools} we introduce a few probabilistic tools that will be helpful later on (see also Appendix~\ref{sec:appendix}).
	In Section~\ref{sec:nibble}--\ref{sec:expectation-concentration}, we give a proof of Lemma~\ref{lem:nibble} and in Section~\ref{sec:finisher} we prove Lemma~\ref{lem:edge-finisher}.
	We finish by stating a few open problems and questions in Section~\ref{sec:open-problems}.

	\section{Proof of main result}
	We start with a sketch of the proof of Theorem~\ref{thm:main}.
	A quick application of the Lov\'asz Local Lemma (see Section~\ref{sec:tools}) shows that we can find the desired colouring, provided that the weighted size of every list $L(e)$ is large enough with respect to the weighted number of possible conflicts $N(e,v,c)$ for each $v \in e$ and $c \in L(e)$.
	This is made precise in the following lemma, whose proof is deferred to Section~\ref{sec:finisher}.
	\begin{lemma}[Finisher]\label{lem:edge-finisher}
		Let $G =(V,E)$ be a $k$-uniform linear hypergraph with an edge correspondence $\sigma$ and weighted lists of colours $(L,\mu)$ assigned to its edges.
		Suppose that the lists have finite size.
		Let $\L,\N >0$ such that, for all $e \in E$, all of the following hold:
		\begin{enumerate}[\upshape (i)]
			\item $\L/\N\geq 3\e k$, 
			\item {$\om{L(e)} \geq \L$},
			\item $\om{N_{G,L,\sigma}(e,v,c)}\leq \N$ for all $v \in e$ and $c \in L(e)$.
		\end{enumerate}
		Then $G$ is $(L,\sigma)$-colourable.
	\end{lemma}
	From the assumptions of Theorem~\ref{thm:main}, it follows that the ratio between the weighted sizes of $L(e)$ and $N(e,v,c)$ is at least $1+\eps$, which is too low to apply Lemma~\ref{lem:edge-finisher} right away.
	The proof of Theorem~\ref{thm:main} therefore begins with a `nibbling' argument, i.e.~an iterative approach, where, in each step, we colour a few further edges improving the above mentioned ratio by a factor of roughly $1+\frac{\eps}{\ln {\delta}}$.
	Hence, after $O(\ln \delta)$ iterations, we can finish the colouring by with Lemma~\ref{lem:edge-finisher}.

	In each iteration, we use the \emph{naive colouring procedure} to find the desired colouring.
	This method consists of two steps:
	\begin{enumerate}[\upshape (I)]
		\item \label{prod:assign-simple} we randomly assign to each edge a small (possibly empty) set of permissible colours from its list independently from other edges, and 
		\item \label{prod:remove-conflict-simple} we resolve conflicts between these assignments by uncolouring some of the edges.
	\end{enumerate}
	
	A concentration analysis shows that with positive probability the resulting $(L, \sigma)$-colouring has the desired properties.  
	The following lemma formalizes these ideas and is proved in Section~\ref{sec:nibble}.
	A \emph{partial} $(L,\sigma)$-colouring of $G$ is an $(L,\sigma)$-colouring of a subgraph of $G$ from the same lists. 
	
	\begin{lemma}[Nibbler]\label{lem:nibble}
		For every $0 < \eps \leq 1/4$ and $k \in \NATS$ there exists $\delta \in \NATS$ with the following properties.
		Let $G =(V,E)$ be a $k$-uniform linear hypergraph with edge correspondence $\sigma$ and weighted lists of colours $(L,\mu)$ assigned to its edges. 	Suppose that for  $\L,\N > \delta$, all of the following hold:
		\begin{enumerate}[\upshape (a)]
			\item \label{itm:bound-weights} $\mu(e)\colon L(e) \to [\exp\left({- \N^{1/20}}\right),1]$,
			\item \label{itm:bound-ratio} $3\e k > \frac{\L}{\N}>  1+\eps $,
			\item \label{itm:bound-list-size} $\om{L(e)}= \L$ for every edge $e  \in E$,
			\item \label{itm:bound-colour-neighbours} $\om{N_{G,L,\sigma}(e,v,c)} \leq \N$  for every edge $e \in E$, vertex $v \in e$ and colour $c \in \NATS$.
		\end{enumerate}
		
		Then there is a partial $(L,\sigma)$-colouring of $G$ with the following properties.
		Let $G'=(V',E') \subset G$ be the subgraph induced by the uncoloured edges.
		For each $e \in E'$, let $L'(e)$ be obtained from $L(e)$ 
		by removing the colours $c$ such that there is an edge $f$ adjacent to $e$ in $G$ such that
		$(e, c)$ is blocked by $(f, c' )$, where $c'$ is the colour used on $f$.
		Then there exist weights $\mu'(e) \colon L'(e) \to (0,1]$ and numbers $\L',\N' >0$ with 
		\begin{enumerate}[\upshape (a$'$)]
			\item \label{itm:nibble-output-L'-N'-lower-bound} $\N' \geq \left(1-\tfrac{5k}{\ln \N}\right) \N$,
			\item \label{itm:nibble-output-ratio-L/N} $\frac{\L'}{\N'} \geq \left(1+\tfrac{\eps}{16\ln \N}\right) \frac{\L}{\N} $
		\end{enumerate}
		such that, for every edge $e \in E'$, vertex $v \in e$ and colour $c \in L'(e)$, all of the following hold:
		\begin{enumerate}[\upshape (a$'$)]
			\addtocounter{enumi}{2}
			\item \label{itm:nibble-output-colour-lists} $|L'(e)|_{\mu'} = \L'$,
			\item \label{itm:nibble-output-enemy-lists} $|N_{G',L',\sigma}(e,v,c)|_{\mu'}\leq \N'$,
			\item \label{itm:nibble-output-mu}$\left(1-\frac{2}{\L}\right) \mu(e,c) \leq \mu'(e,c) \leq \mu(e,c)$.
		\end{enumerate}
	\end{lemma}
	Now we are ready to prove Theorem~\ref{thm:main}.
	
	\begin{proof}[Proof of Theorem~\ref{thm:main}]
		Let $\delta_{L.\ref{lem:nibble}}=\delta_{L.\ref{lem:nibble}}(\eps,k)$ be obtained from Lemma~\ref{lem:nibble} with input $(\eps,k)$ where we assume without loss of generality that $\delta_{L.\ref{lem:nibble}}$ is large enough so that $\frac{5k}{\ln \delta_{L.\ref{lem:nibble}}} \le \frac{1}{2}$ (i.e.~$\delta_{L.\ref{lem:nibble}} \ge e^{10k}$) and $400\eps^{-1}k\ln \delta_{L.\ref{lem:nibble}} \leq \delta_{L.\ref{lem:nibble}}$.
		We choose $\delta= \delta (\eps,k)$ sufficiently large such that in particular $\delta \geq \e^{1000k^2\eps^{-1}} \delta_{L.\ref{lem:nibble}}$.
		Let $G$, $\sigma$ and $(L,\mu)$ be as in the statement of Theorem~\ref{thm:main}.
		
		We will define, for each $0 \leq i \leq 100\eps^{-1}k  \ln \delta_{L.\ref{lem:nibble}}$, a partial $(L,\sigma)$-colouring $\gamma_i$ of $G$ such that for the graph $G_i=(V_i,E_i)$ of edges not coloured by $\gamma_i$ the following holds.
		There exist parameters $\L_{i}, \N_{i}$ and weighted lists of colours $(L_{i},\mu_i)$ for $G_i$ such that, for every edge $e \in E_i$, vertex $v \in e$ and colour $c \in L_i(e)$, all of the following hold:
		\begin{enumerate}[\upshape (a$''$)]
			\item \label{itm:main-thm-n-bound} $\delta \geq \N_i \geq \left(1-\tfrac{5k}{\ln \delta_{L.\ref{lem:nibble}}}\right)^i \delta$, 
			\item \label{itm:proof-main-thm-ratio} $\frac{\L_i}{\N_i} \geq  \min\left\{ 3\e k,\left(1+\tfrac{\eps}{16\ln \delta}\right)^i \left(1+\eps\right)\right\}$, 
			\item \label{itm:main-thm-list-bound} $|{L_i(e)} |_{\mu_i} = \L_i$, 
			\item \label{itm:main-thm-deg-bound}  $|{N_{G_i,L_i,\sigma}(e,v,c)} |_{\mu_i}\leq \N_i$, 
			\item  \label{itm:main-thm-mu-bound} $\mu_i(e)\colon L_i(e)\to  \left[\left(1-\frac{2}{\delta_{L.\ref{lem:nibble}}}\right)^{i+1} \exp\left(-\delta^{1/24}\right),1\right]$.
		\end{enumerate}
		
		{To show that this is possible, we proceed inductively.
			For $i =0$, set $G_0 = G$, $\gamma_0 = \es$, $\L_0=(1+\eps)\delta$ and $\N_0 = \delta$. 
			Note that~\ref{itm:main-thm-n-bound} and~\ref{itm:proof-main-thm-ratio} hold trivially.
			By assumption, we have $|L(e)|_{\mu'} \geq \L_0$, whereas part~\ref{itm:main-thm-list-bound} requires equality.
			We guarantee equality by truncating the lists $L(e)$ and (mildly) rescaling the weight function $\mu' := \delta \mu$ as follows.  
			For every edge $e \in E$, let $L_0(e)$ be a copy of $L(e)$.
			We truncate $L_0(e)$ as follows.
			While there is a colour $c \in L_0(e)$ with $|L_0(e)\sm \{c\}|_{\mu'} \geq \L_0$, we delete $c$ from $L_0(e)$.
			As $\mu'  \leq 1$, we have $\L_0 \leq |L_0(e)|_{\mu'}  < \L_0+1$ after this procedure.
			We then scale the weights by setting $\mu_0(e,c) = \frac{\L_0}{|{L_0(e)|_{\mu'}}} \mu'(e,c)$ for every colour $c \in L_0(e)$.
			It follows that $|L_0(e)|_{\mu_0} = \L_0$ as required for~\ref{itm:main-thm-list-bound}.
			Moreover, we have
			\begin{align}\label{equ:truncate}
				\left(1-\frac{2}{\L_0}\right) \mu'(e,c) \leq \left(\frac{\L_0}{\L_0+1}\right) \mu'(e,c) \leq \mu'(e,c) \leq \mu_0(e,c)
			\end{align}
			for every uncoloured edge $e \in E$  and colour $c \in L_0(e)$.
			Now,~\ref{itm:main-thm-deg-bound} holds since Theorem~\ref{thm:main}\ref{itm:con-main-thm-N} holds by assumption and the right side of \eqref{equ:truncate}. Finally,~\ref{itm:main-thm-mu-bound} holds since $\delta\ge 1$ and since $\delta_0 \in \mathbb{N}$, since Theorem~\ref{thm:main}\ref{itm:con-main-thm-mu} holds by assumption and by the left side of \eqref{equ:truncate}.}

		Now suppose \ref{itm:main-thm-n-bound}--\ref{itm:main-thm-mu-bound} hold for $i$ with $0 \leq i < 100\eps^{-1}k  \ln \delta_{L.\ref{lem:nibble}}$.
		In the following, we use repeatedly the fact that $1-x\geq\exp(-x/(1-x))$ for $0 \leq x <1.$ 
		It follows that 
		$$\L_i \overset{\text{\ref{itm:proof-main-thm-ratio}}}{\geq} \N_{i} \overset{\text{\ref{itm:main-thm-n-bound}}}{\geq} \left(1-\tfrac{5k}{\ln \delta_{L.\ref{lem:nibble}}}\right)^{100 \eps^{-1}k \ln \delta_{L.\ref{lem:nibble}}} \delta \geq  \e^{-1000k^2\eps^{-1}} \delta \ge \delta_{L.\ref{lem:nibble}},$$
		where we used that $\frac{5k}{\ln \delta_{L.\ref{lem:nibble}}} \le \frac{1}{2}$ and that $\delta \geq \e^{1000k^2\eps^{-1}} \delta_{L.\ref{lem:nibble}}$. Similarly
		\begin{align*}
			\mu_i(e,c) & \overset{\text{\ref{itm:main-thm-mu-bound}}}{\geq} \left(1-\frac{2}{\delta_{L.\ref{lem:nibble}}}\right)^{100\eps^{-1}k\ln \delta_{L.\ref{lem:nibble}}} \exp\left(-\delta^{1/24}\right) 
			\\
			& \geq \exp\left(-\frac{400\eps^{-1}k\ln \delta_{L.\ref{lem:nibble}}}{\delta_{L.\ref{lem:nibble}}} \right) \exp\left(-\delta^{1/24}\right)
			\\
			& \geq \exp\left(-1 \right) \exp\left(-\delta^{1/24}\right)
			\\
			&\geq \exp\left(-\delta^{1/20}\right),
		\end{align*}
		where for the second inequality we used that $\frac{2}{\delta_{L.\ref{lem:nibble}}} \le \frac{1}{2}$, for the third inequality we used that $400\eps^{-1}k\ln \delta_{L.\ref{lem:nibble}} \leq \delta_{L.\ref{lem:nibble}}$, and for the fourth inequality we used that $1+\delta^{1/24} \le \delta^{1/20}$ since $\delta$ is large enough.

		Note that, if  $\L_i/\N_i \geq 3\e k$, there is nothing to do.
		So assume that $\L_i/\N_i \leq  3\e k$. 
		Now we apply Lemma~\ref{lem:nibble} with
		\begin{center}
			\begin{tabular}{ c  *{5}{|c} }
				object/parameter & $G_{i}$ & $L_{i}(e)$  & $\mu_{{i}}$  & $\L_{i}$ &  $\N_{i}$ \\
				\hline
				playing the role of	 & $G$ & $L(e)$  & $\mu$   & $\L$  &  $\N$ 
			\end{tabular}
		\end{center}
		to obtain a partial $(L'_i,\sigma)$-colouring $\gamma_{i+1}'$ of $G_{i}$, parameters $\L_{i+1}, \N_{i+1}$ and weighted lists of colours $(L_{i+1},\mu_{i+1})$ that satisfy (a$''$)--(e$''$).
		Note in particular, that we obtain (a$''$) since $n_i \geq \delta_{L.\ref{lem:nibble}}$, (b$''$) since $n_i \leq \delta$, and (e$''$) since $\L_i \geq \delta_{L.\ref{lem:nibble}}$.
		We then let $\gamma_{i+1}$ be the union of the colourings $\gamma_i$ and $\gamma_{i+1}'$.
		
		Finally, observe that by (b$''$), we have $\L_{i^*}/\N_{i^*} \geq 3\e k$ for some $i^* \leq 100\eps^{-1}k  \ln\delta$.
		Thus we may finish the proof of the theorem by applying Lemma~\ref{lem:edge-finisher}.
	\end{proof}
	
	\section{Tools}\label{sec:tools}
	In this section, we collect some of the tools that will be used in the proofs.
	We start with the following simple lemma.
	\begin{lemma}\label{lem:weighted-binomial-coeff-bound}
		Let $p_1,\ldots,p_N >0$ such that $\sum_{i =1}^N p_i \leq p$.
		Then we have
		\begin{align*}
			\sum_{S \in \binom{N}{k}} \prod_{i \in S} p_i \leq \left(\frac{\e p}{k}\right)^k.
		\end{align*}
	\end{lemma}
	\begin{proof}
		Note that for every $S \in \binom{N}{k}$, the product $\prod_{i \in S} p_i$ appears $k!$ times in the sum obtained from expanding $(p_1+ \cdots+p_N)^k$.
		Moreover, by the Taylor expansion of the exponential function, we have $\e^k \geq \frac{k^k}{k!}$.
		It follows that
		\begin{align*} 
			\sum_{S \in \binom{N}{k}} \prod_{i \in S} p_i 
			\leq \frac{(p_1+\cdots+p_N)^k }{k!} 
			= \frac{p^k}{k!} \leq \left(\frac{\e p}{k}\right)^k.
		\end{align*}
	\end{proof}

	We will use the following corollary of Talagrand's Inequality.
	A very similar version of this theorem has been derived by Molloy and Reed~\cite[Talagrand's Inequality V]{MR13}. We note that the proof is the straightforward generalization of the proof of Theorem 6.3 in~\cite{KP20a} but we include a proof in Appendix~\ref{sec:appendix} for completeness. But first a definition.
	
	\begin{definition}
		Let $((\Omega_i, \Sigma_i, \mathbb P_i))_{i=1}^n$ be probability spaces, let $(\Omega, \Sigma, \mathbb P)$ be their product space, let $\Omega^* \subseteq \Omega$ be a set of \textit{exceptional outcomes}, and let $X : \Omega \rightarrow \mathbb R_{\geq0}$ be a non-negative random variable.  Let $b > 0$.
		\begin{itemize}[(1)]
			\item If $\omega = (\omega_1, \dots, \omega_n) \in \Omega$ and $s \ge 0$, a \textit{$b$-certificate} for $X, \omega, s$, and $\Omega^*$ is an index set $I\subseteq\{1, \dots, n\}$ and a vector $(c_i: i\in I)$ with 
			$$\sum_{i\in I} c_i^2 \le bs$$ 
			such that for all $I'\subseteq I$, we have that
			\begin{equation*}
				X(\omega') \geq s - \sum_{i\in I'} c_i,
			\end{equation*}
			for all $\omega' = (\omega'_1, \dots, \omega'_n)\in\Omega\setminus\Omega^*$  such that $\omega_i=\omega_i'$ for all $i\in I\setminus I'$.
			\item If for every $\omega\in\Omega\setminus\Omega^*$, there exists a $b$-certificate for $X, \omega, s:=X(\omega)$, and $\Omega^*$, then $X$ is \textit{$b$-certifiable} with respect to $\Omega^*$.
		\end{itemize}
	\end{definition}
	
	\begin{theorem}\label{thm:talagrand-V}
		Let $((\Omega_i, \Sigma_i, \mathbb P_i))_{i=1}^n$ be probability spaces, let $(\Omega, \Sigma, \mathbb P)$ be their product space, let $\Omega^* \subseteq \Omega$ be a set of exceptional outcomes, and let $X : \Omega \rightarrow \mathbb R_{\geq0}$ be a non-negative random variable.  Let $b > 0$.
		
		If $X$ is $b$-certifiable with respect to $\Omega^*$, 
		then for any $t > 96\sqrt{b\Expect{X}} +  128b + 8\Prob{\Omega^*}(\sup X),$
		\begin{equation*}
			\Prob{|X - \Expect{X}| > t} \leq 4\exp\left({\frac{-t^2}{8b(4\Expect{X} + t)}}\right) + 4\Prob{\Omega^*}.
		\end{equation*}
	\end{theorem}
	
	We will also need the (general) Lov\'asz Local Lemma.
	Let $B_1,B_2,\dots,B_n$ be events in an arbitrary probability space.
	We think of the $B_i$'s as `bad' events and want to show that none of them occur with some positive probability.
	The Lov\'asz Local Lemma tells us that this is feasible provided that the interdependence of these events is sufficiently bounded.
	This is formalized as follows:
	a directed graph $D=(V,E)$ on the set of vertices $V=\{1,\ldots,n\}$ is called a \emph{dependency digraph} for the events $B_1,\ldots,B_n$ if for each $i$, $1 \leq i \leq n$, the event $B_i$ is mutually independent of all the events $\{B_j:\,(i,j) \notin E\}$.
	\begin{theorem}[Lov\'asz Local Lemma~\cite{EL75}]\label{thm:local-lemma}
		Let $B_1,B_2,\dots,B_n$ be events in an arbitrary probability space.
		Suppose that $D=(V,E)$ is a dependency digraph for these events and suppose that there are real numbers $x_1,\dots,x_n$ such that $0\leq x_i < 1$ and $\Pr(B_i)\leq x_i\prod_{(i,j)\in E}(1-x_j)$.
		Then
		$$\Pr\left( \bigwedge_{i=1}^n \overline{B_i} \right) \geq \prod_{i=1}^n (1-x_i) >0.$$
	\end{theorem}
	The following result can be derived from Theorem~\ref{thm:local-lemma} by taking $x_i = 1/(d+1)$ for all $i$ and using that $1-x \geq \exp(-{x}/{(1-x)})$ for $x<1$.
	\begin{corollary}[Local Lemma; Symmetric Case]\label{cor:simple-local-lemma}
		Let $B_1,B_2,\dots,B_n$ be events in an arbitrary probability space.
		Suppose that each event $B_i$ is mutually independent of all but at most $d$ other events, and that $\Pr(B_i) \leq p$ for all $1 \leq i \leq n$.
		If $\e p (d+1) \leq 1$, then $\Pr\left( \bigwedge_{i=1}^n \overline{B_i} \right) >0$.
	\end{corollary}
	
	\section{The naive colouring procedure}\label{sec:nibble}
	This section is dedicated to the proof of Lemma~\ref{lem:nibble}.
	Before we dive into the details, let us lay out the general strategy.
	As mentioned above, we will use the naive colouring procedure.
	Our actual procedure consists of three steps, \ref{prod:assign} assigning colours, \ref{prod:remove-conflict} resolving conflicts, and \ref{prod:coinflip} performing an {equalizing coin flip}.
	Let $L'(e)$ be the list of remaining colours after the procedure.
	We note that during steps \ref{prod:remove-conflict} and \ref{prod:coinflip}, we will delete more colours from $L'(e)$ than necessary, making the procedure `{wasteful}'.
	Moreover, our procedure permits for multiple colour assignments to an edge.
	This allows us to simplify the analysis of the procedure at cost of marginally worse parameters.
	In a similar vein, the purpose of the coin flip, is to guarantee that the probability that a colour $c \in L(e)$ will be in $L'(e)$ is \emph{uniformly} $\K^k$ for all edges $e$ and colours $c$, where $\K  \approx 1-\frac{1}{(1+\eps)\ln \N}$.
	(Without a coin flip these probabilities would differ between the edges, making the analysis more cumbersome.)
	A short argument shows that the expected value of the random variable $\om{L'(e)}$ is $\L \cdot \K^k$.
	By showing that  $\om{L'(e)}$ is concentrated around its expected value, we deduce that with high probability
	$$ \om{L'(e)}  \approx \L \cdot \K^k.$$
	A similar argument yields that, for a fixed edge $e$, vertex $v \in e$ and colour $c \in L(e)$, with high probability we can bound the weight of the colour neighbours after the procedure by
	$$\om{N'(e,v,c)} \lessapprox  \N \cdot  \K^{k} \cdot \left(1-\frac{\K^{k}}{\ln \N} \right) .$$ 
	Using the Local Lemma, we show that with positive probability these bounds hold simultaneously for all edges, vertices and colours.
	Finally, computations under these assumptions give that $$\frac{\om{L'(e)}}{\om{N'(e,v,c)}} \geq  \left(1+\tfrac{\eps}{16\ln \N}\right) \frac{\L}{\N},$$ as desired.
	Now come the details.
	
	\begin{proof}[Proof of Lemma~\ref{lem:nibble}]
		Given $\eps>0$ and $k \in \NATS$, we choose $\delta = \delta(\eps,k) \in \NATS$ sufficiently large.
		Let $G =(V,E)$, $L(e)$, $\sigma$, $\mu$, $\L$, $\N$ be as in the statement of Lemma~\ref{lem:nibble}.
		For every edge $e \in E$, vertex $v \in e$ and colour $c \in L(e)$, denote $N(e,v,c) = N_{G,L,\sigma}(e,v,c)$.
		
		In order to define $\L'$ and $\N'$, we need to set up some intermediate parameters.
		Let 
		{$$\K = 1-{\frac{\N}{\L}} \frac{1}{\ln \N} .$$}
		For an edge $e \in E$, vertex $v \in e$ and colour $c \in \NATS$, we define
		\begin{align*}
			\Eq(e,v,c) &= \frac{\K }{\prod_{(f,c') \in N(e,v,c)} \left(1-\frac{\mu(f,c')}{\L}\frac{1}{\ln \N}\right)}.
		\end{align*}
		{Recall that $\prod_i (1 - a_i) \geq 1 - \sum_i a_i$ for $a_i \in [0, 1]$.
			Together with assumption~\ref{itm:bound-colour-neighbours} of Lemma~\ref{lem:nibble}, this gives
			\begin{align*}
				\prod_{(f,c') \in N(e,v,c)} \left(1-\frac{\mu(f,c')}{\L}\frac{1}{\ln \N}\right) 
				\geq 1- \sum_{(f,c') \in N(e,v,c) }  \frac{\mu(f,c')}{\L}\frac{1}{\ln \N}
				\geq 1-\frac{\N}{\L}\frac{1}{\ln \N}.
		\end{align*}}
		So in particular, $0 \le \Eq(e,v,c) \leq 1$.
		Now, let
		\begin{align*}
			\L' &= \L \cdot \K^k - \N^{2/3};
			\\ \N' &= \N \cdot  \K^{k-1} \cdot \left(1-\frac{{1-\eps/2}}{\ln \N} \K^{k}  \right) + \N^{2/3}.
		\end{align*}
		
		In the following, we show that Lemma~\ref{lem:nibble} holds with $\L'$ and $\N'$.
		The next claim covers properties~\ref{itm:nibble-output-L'-N'-lower-bound}--\ref{itm:nibble-output-ratio-L/N} of Lemma~\ref{lem:nibble}.
		Its proof can be found in Section~\ref{sec:ratio}.
		\begin{claim}\label{cla:ratio}
			We have $\L' \geq  (1-\frac{2k}{\ln \N})\L$, $\N' \geq  (1-\frac{3k}{\ln \N})\N$, and $\L'/\N' \geq (1+ \frac{\eps}{16 \ln \N})(\L/\N)$.
		\end{claim}
		
		We use the following random procedure to colour some edges of $G$.
		\begin{definition}[Random colouring procedure]
			Initialize $L'(e)$  as a copy of $L(e)$ for each $e \in E$.
			\begin{enumerate}[(I)]
				\item \label{prod:assign}  For every colour $c \in L(e)$ and edge $e \in E$, assign $c$ to $e$ with probability $\frac{\mu(e,c)}{\L} \frac{1}{\ln \N}$ independently of all other assignments.
				\item \label{prod:remove-conflict} For each edge $e$, vertex $v \in e$, colour $c \in L(e)$ and every pair $(f,c') \in N(e,v,c)$, do the following. 
				If $c'$ is assigned to $f$, then
				\begin{enumerate}
					\item remove $c$ from $L'(e)$, and
					\item if $c$ was assigned to $e$, remove $c$ from $e$.
				\end{enumerate}
				\item  \label{prod:coinflip}  For every edge $e \in E$, vertex $v \in e$ and colour $c \in L(e)$, perform an independent coin flip $F(e,v,c)$ that returns $1$ with probability $\Eq(e,v,c)$.
				If the $F(e,v,c)$ returns $0$, then
				\begin{enumerate}
					\item remove $c$ from $L'(e)$, and
					\item if $c$ was assigned to $e$, remove $c$ from $e$.
				\end{enumerate}
			\end{enumerate}
		\end{definition}

		Let $G'=(V',E') \subset G$ be the subgraph of uncoloured edges after step~\ref{prod:coinflip} of the procedure.
		We say that a colour $c \in L(e)$ was \emph{removed} from $L'(e)$ if $c \notin L'(e)$ after step~\ref{prod:coinflip} of the procedure.
		For every edge $e \in E'$, vertex $v \in e$ and colour $c \in L'(e)$, let $N'(e,v,c) = N_{G',L',\sigma}(e,v,c)$.
		The next claim covers~\ref{itm:nibble-output-colour-lists}--\ref{itm:nibble-output-enemy-lists} of Lemma~\ref{lem:nibble}.
		Its proof can be found in Section~\ref{sec:bounding-colour-and-enemy-lists}.
		\begin{claim}\label{cla:main}
			With positive probability we have
			\begin{enumerate}[\upshape (a)]
				\item $\om{L'(e)} \geq \L'  $ for every edge $e  \in E'$, and  
				\item $\om{N'(e,v,c)}  \leq \N'  $  for every edge $e \in E'$, vertex $v \in e$ and colour $c \in \NATS$.
			\end{enumerate}
		\end{claim}
		Suppose that the edges of $G$ have been coloured such that the properties of Claim~\ref{cla:main} are satisfied.
		To finish the proof we require equality between $\om{L'(e)}$ and $\L'$.
		Fortunately, this can be easily arranged by truncating the lists $L'(e)$ and scaling the weight function $\mu$.
		For every uncoloured edge $e \in E$, let $L''(e)$ be a copy of $L'(e)$.
		We truncate $L''(e)$ as follows.
		While there is a colour $c \in L''(e)$ with $\om{L(e)\sm \{c\}} \geq \L'$, we delete $c$ from $L''(e)$.
		As $\mu \leq 1$, we have $\L' \leq \om{L''(e)} < \L'+1$ after this procedure.
		We then scale the weights by setting $\mu'(e,c) = \frac{\L'}{|{L''(e)|_{\mu}}} \mu(e,c)$ for every colour $c \in L''(e)$.
		It follows that $|L''(e)|_{\mu'} = \L'$ and 
		$$\left(1-\frac{2}{\L}\right) \mu(e,c) \leq \left(\frac{\L'}{\L'+1}\right) \mu(e,c) \leq \mu'(e,c) \leq \mu(e,c)$$
		for every uncoloured edge $e \in E$  and colour $c \in L'(e)$.
		This yields~\ref{itm:nibble-output-mu} of Lemma~\ref{lem:nibble} and finishes its proof.
	\end{proof}

	\section{Ratio} 
	\label{sec:ratio}
	\begin{proof}[Proof of Claim~\ref{cla:ratio}]
		By the assumptions of Lemma~\ref{lem:nibble}, we have $\L/\N \geq (1+\eps)$ and  $\eps \leq 1/4$.
		This gives
		\begin{equation}\label{equ:Keep>1-1/(1+eps)lnD}
			1 \geq \K = 1-\frac{\N}{\L}  \frac{1}{\ln \N}  \geq {1- \frac{1}{1+\eps}  \frac{1}{\ln \N}} \geq 1- \frac{1-3\eps/4}{\ln \N}  .
		\end{equation}
		In particular, for $\N$ large enough with respect to $k$, this implies that $\K^k \geq 1-\frac{k}{\ln \N} \ge \frac{1}{2} \ge \N^{-1/6}$.
		So
		\begin{align}\label{equ:L'-lower-bound}
			\L' &= \L \cdot \K^k - \N^{2/3} > \L \cdot \K^k\cdot (1- \N^{-1/5}),
		\end{align}
		where we used that $\K^k \ge \N^{-1/6} \ge \frac{\N^{5/6}}{\L}$ since $\N \le \L$. Furthermore,
		\begin{align}\label{equ:N'-upper-bound}
			\N' &=  \N \cdot \K^{k-1} \cdot \left(1-\frac{1-\eps/2}{\ln \N} \K^{k} {    }\right) + \N^{2/3}
			\\ &<  \N \cdot \K^{k-1} \cdot \left(1-\frac{1-\eps/2}{\ln \N} \K^{k} \right) \cdot (1+\N^{-1/5}).
		\end{align}
		
		Note that~\eqref{equ:L'-lower-bound} gives 	$\L' \ge \left(1-\frac{2k}{\ln \N}\right) \L$.
		We also have 
		\begin{align*}
			\N' &=  \N \cdot \K^{k-1} \cdot \left(1-\frac{1-\eps/2}{\ln \N} \K^{k}  \right) + \N^{2/3}
			\\ &>  \N \cdot \left(1-\frac{k}{\ln \N}\right)^2 >  \N \cdot \left(1-\frac{3k}{\ln \N}\right).
		\end{align*}
		For $\N$ sufficiently large, it follows by and~\eqref{equ:Keep>1-1/(1+eps)lnD},~\eqref{equ:L'-lower-bound}, and~\eqref{equ:N'-upper-bound} that
		\begin{align*}
			\frac{\L'}{\N'} &\geq \frac{\L}{\N}\cdot \frac{\K}{{1-\frac{1-\eps/2}{\ln \N}\K^{k}}} \cdot\frac{1-\N^{-1/5}}{1+\N^{-1/5}}
			\\ &\overset{}{>} \frac{\L}{\N}
			\cdot \left(1- \frac{1-3\eps/4}{\ln n} \right)
			\cdot \left(1+ \frac{1-\eps/2}{\ln \N} \K^{k}\right) 
			\cdot (1-2\N^{-2/5}) 
			\\ &> \frac{\L}{\N}
			\cdot \left(1- \frac{1-3\eps/4}{\ln n} \right)
			\cdot \left(1+ \frac{1-2\eps/3}{\ln \N} \right) 
			\\ &\geq \frac{\L}{\N} 
			\cdot \left(1   + \frac{\eps}{16 \ln \N} \right).
		\end{align*}

	\end{proof}
	
	\section{Proof of Claim~\ref{cla:main}}\label{sec:bounding-colour-and-enemy-lists}\label{sec:expectation-concentration}
	In this section we prove Claim~\ref{cla:main}.
	Our approach follows a concentration argument.
	For $\om{L'(e)}$, we will first bound the expected value of the random variable $\om{L'(e)}$ and then show that it is highly concentrated around its expectation.
	In the case of $\om{N'(e,v,c)}$, we have to be a bit more careful.
	Observe that assigning some colour $c'$ to an edge $f \sim v$ might change $\om{N'(e,v,c)}$ significantly, since $(f,c')$ could block many pairs $(g,\alpha) \in N(e,v,c)$.
	So in fact, $\om{N'(e,v,c)}$ is not concentrated.
	Thus, instead of computing the expected value of $\om{N'(e,v,c)}$, we focus on a different random variable that does not take into account edges $f$ incident to $v$.
	
	Recall that $E'$ is the set of uncoloured edges after the colouring procedure.
	Consider an edge $e \in E'$, a vertex $v \in e$ and a colour $c \in L(e)$.
	We say that $L'(e)$ \emph{loses $c$ at $v$}, if 
	\begin{itemize}
		\item $f$ is assigned $c'$ for some $(f,c') \in N(e,v,c)$, or
		\item the coin flip $F(e,v,c)$ returns $0$.
	\end{itemize}
	Note that the probability that $L'(e)$ loses $c$ at $w$ is $1-\K$.
	Let $N''(e,v,c)$ be the set of pairs $(f,c') \in N(e,v,c)$ such that
	\begin{enumerate}[(i)]
		\item \label{itm:loosing-all-of-f} $f$ does not retain a colour after the procedure, and
		\item \label{itm:not-loosing-c'} $f$ does not lose $c'$ at any $w \in e \sm \{v\}$.
	\end{enumerate}
	Note that $\om{N'(e,v,c)} \leq \om{N''(e,v,c)}$.
	So it suffices to carry out the analysis for $N''(e,v,c)$.
	The following claim deals with the expectations.
	
	\begin{claim}\label{cla:expectation} For every edge $e \in E'$, vertex $v \in e$ and colour $c \in \NATS$, we have
		\begin{enumerate}[\upshape (a)]
			\item \label{itm:expectation-L} $\Exp\left({\om{L'(e)}}\right) =   \L \cdot \K^k$, and
			\item \label{itm:expectation-N} $\Exp \left({\om{N''(e,v,c)}}\right) \leq \N\cdot \K^{k-1} \cdot \left(1-\frac{{1-\eps/2}}{\ln \N} \K^{k} \right) = {n' - n^{2/3}}.$ 
		\end{enumerate}
	\end{claim}
	The next claim bounds the probability that the random variables $\om{L'(e)}$ and $\om{N'(e,v,c)}$ deviate from their expectation for fixed choices of $e,v,$ and $c$.
	\begin{claim}\label{cla:concentration}
		For every edge $e \in E'$, vertex $v \in e$ and colour $c \in L(e)$ it holds with probability at least $1- 16\cdot  {\exp\left({ -\N^{1/10}}\right)}$ that
		\begin{enumerate}[\upshape (a)]
			\item \label{itm:concentration-colour-list} $\om{L'(e)} \geq \L'$, and  
			\item \label{itm:concentration-enemy-list}  $\om{N''(e,v,c)} \leq \N'$.
		\end{enumerate}
	\end{claim}
	Before we prove Claim~\ref{cla:expectation} and~\ref{cla:concentration}, let us show how they imply Claim~\ref{cla:main}.
	\begin{proof}[Proof of Claim~\ref{cla:main}]
		For an edge $e \in E$, let $N^0(e)= \{e\}$.
		We define the edges of \emph{colour-distance at most $\ell$ to $e$} as
		$$
		N^\ell(e) = 
		\bigcup_{f \in N^{\ell-1}(e)}
		\bigcup_{v \in f}
		\bigcup_{c \in L(f)}
		\bigcup_{(g,c') \in N(f,v,c)} \{g\}.$$
		
		Recall that by assumption~\ref{itm:bound-weights} of Lemma~\ref{lem:nibble}, we have  $\mu\colon E \times \NATS \to [\exp(-\N^{1/20}),1]$.
		So, for every set of pairs $A \subset E \times \NATS$, we have $|A| \leq \exp\left(\N^{1/20}\right) \om{A}.$ 
		Hence assumptions~\ref{itm:bound-list-size} and~\ref{itm:bound-colour-neighbours} of Lemma~\ref{lem:nibble}, allow us to  bound
		$$
		|N^d(e)| \leq \left( k \cdot \frac{\L}{\exp(-\N^{1/20})} \cdot  \frac{\N}{\exp(-\N^{1/20})}\right)^d.
		$$

		For every edge $e \in E'$, vertex $v \in e$ and colour $c \in L(e)$, we define $B(e)$ to be the event that $\om{L'(e)} < \L'$ and $B(e,v,c)$ to be the event that $\om{N'(e,v,c)} > \N'$.
		Each event $B(e)$ and $B(e,v,c)$ is determined by the colour assignments and coin flips to edges of colour-distance at most $1$ and $2$, respectively.
		Thus every event is mutually independent of all but at most $(k\L \N)^4 \exp\left(8\N^{1/20}\right)$ other events.
		Since, for $\N$ large enough, we have
		$$\e \cdot 16\cdot {\exp({-8 \N^{1/10} })} \cdot \left((k\L \N)^4 \exp\left(8\N^{1/20}\right) +1\right)\leq 1,$$
		and the claim follows by Claim~\ref{cla:concentration} and Corollary~\ref{cor:simple-local-lemma}.
	\end{proof}
	
	It remains to prove Claim~\ref{cla:expectation} and Claim~\ref{cla:concentration}.
	
	\subsection{Expectation}
	In this subsection, we show  Claim~\ref{cla:expectation}.
	We start by proving part~\ref{itm:expectation-L} of the claim.
	Note that for an edge $e$, vertices $v,w \in e$ and colours $c,c' \in L(e)$, the sets $N(e,v,c)$ and $N(e,w,c')$ are disjoint as $G$ is a linear hypergraph.
	Moreover, since the correspondences are integer permutations, assigning a colour to an edge incident to $v$ affects at most one colour of $L(e)$.
	It follows that, for each $v \in e$ and $c \in L(e)$, the probability that $(e,c) \in L'(e)$ after step~\ref{prod:remove-conflict} (but before step~\ref{prod:coinflip}) of the procedure is
	\begin{align*}
		\prod_{v \in e} \prod_{(f,c') \in N(e,v,c)}  \left(1-\frac{\mu(f,c')}{\L}\frac{1}{\ln \N}\right).
	\end{align*}
	Thus 
	\begin{equation}\label{equ:probability-keep}
		\text{the probability that $(e,c) \in L'(e)$ after step~\ref{prod:coinflip}  of the procedure is $\K^k$.}
	\end{equation}
	The part~\ref{itm:expectation-L} of Claim~\ref{cla:expectation} then follows as  
	\begin{equation*}
		\Exp\left({\om{L'(e)}}\right) = \sum_{c \in L(e) } \mu(e,c) \cdot \K^k =  \L \cdot \K^k.
	\end{equation*}
	
	Next, we prove part~\ref{itm:expectation-N} of Claim~\ref{cla:expectation}.
	For an edge $f$, vertex $w \in f$ and colour $\alpha \in L(f)$, let $A(f,\alpha)$ be the event that $f$ is assigned colour $\alpha$ and let $B(f,w,\alpha)$ be the event that $L(f)$ loses $\alpha$ at $w$.
	We claim that
	\begin{equation}\label{equ:independency}
		\text{the events $\{A(f,\alpha), B(f,w,\alpha):\, w \in e \text{ and } \alpha \in L(f) \}$ are mutually independent.}
	\end{equation}
	Indeed, the events $A(f,\alpha)$ are independent of each other by the way we carry out step~\ref{prod:assign} of the colouring procedure.
	The event $B(f,w,\alpha)$ is determined by the colour assignment to edges $g \sim f$ with $w \in g$ and the coin flip $F(f,\alpha)$.
	Note that since $G$ is a linear hypergraph, there is no edge $g \sim f$ that contains two distinct vertices $w,w' \in f$.
	Moreover, as the correspondence $\sigma_{f,g}\colon \NATS \to \NATS$ is a bijection, each colour assignment to $g \sim f$ is relevant to at most one event $B(f,w,\alpha)$,
	with $\alpha \in L(f)$.
	This shows~\eqref{equ:independency}.
	
	Fix $(f,c') \in N(e,v,c)$.
	By~\eqref{equ:independency}, it {follows that $\Pr\left((f,c') \in N''(e,v,c)\right)$ is equal to}
	\begin{align*}
		\K^{k-1} 
		\cdot \left(1-\frac{\mu(f,c')}{\L }\frac{1}{  \ln \N}  \K^{k-1}\right) \cdot \prod_{\alpha \in L(f) \sm \{c'\}} \left(1-\frac{\mu(f,\alpha)}{\L} \frac{1}{\ln \N}\K^{k}\right).
	\end{align*}
	Let us explain this in more detail.
	The factor $\K^{k-1}$ in the first line corresponds to the  probability that $L'(f)$ does not lose $c'$ at any $w \in f\sm\{v\}$.
	Next, the probability that $c'$ is assigned to $f$ and $f$ does not lose $c'$ at any $w \in f \sm \{v\}$ is $\frac{\mu(f,c')}{\L }\frac{1}{  \ln \N} \K^{k-1}$.
	The factor $1-\frac{\mu(f,c')}{\L }\frac{1}{  \ln \N} \K^{k-1}$ corresponds to the probability of the complement of this event.
	Finally, for any $\alpha \in L(f) \sm \{c'\}$, the probability that $\alpha$ is assigned to $f$ and $f$ retains $\alpha$ is $\frac{\mu(f,\alpha)}{\L }\frac{1}{  \ln \N} \K^{k}$.
	So the factor $\prod_{\alpha \in L(f) \sm \{c'\}} \left(1-\frac{\mu(f,\alpha)}{\L} \frac{1}{\ln \N}\K^{k}\right)$ corresponds to the probability that none of these events happen.
	{Given this, we may estimate
		\begin{align*}
			\Pr\left((f,c') \in N''(e,v,c)\right) &\leq {\K^{k-1} \cdot \prod_{\alpha \in L(f)} \left(1-\frac{\mu(f,\alpha)}{\L} \frac{1}{\ln \N}\K^{k}\right)}
			\\&\leq {\K^{k-1} \cdot \exp\left({-\sum_{\alpha \in L(f)}  \frac{\mu(f,\alpha)}{\L} \frac{1}{\ln \N}\K^{k} }\right)}
			\\&= \K^{k-1} \cdot \exp\left({-   \frac{1}{\ln \N}\K^{k} }\right) 
			\\&\leq {\K^{k-1} \cdot \left({1 + \frac{1}{\ln \N}\K^{k} }\right)^{-1}}
			\\&< \K^{k-1} \cdot \left(1-\frac{1-\eps/2}{\ln \N} \K^{k}  \right),
		\end{align*}
		where we used $\K \leq 1$ in the first line, $\e^{x} \geq 1+x$ in the second and fourth line, and  the fact that $n$ is sufficiently large in the last line.}
	This finishes the proof of Claim~\ref{cla:expectation}.

	\subsection{Concentration}
	This subsection is dedicated to the proof of Claim~\ref{cla:concentration}. 
	We start with part~\ref{itm:concentration-colour-list} of the claim.
	Fix an edge $e \in E$.
	Since the correspondences are integer permutations, the event that $c \in L'(e)$ is independent of the event that $c' \in L'(e)$ for distinct colours $c,c'$.
	This and~\eqref{equ:probability-keep} imply that $X = \om{L'(e)}$ is determined by the outcomes of a family of Bernoulli variables $(T_c)_{c \in L(e)}$.
	Given an outcome $x$ of these trials, we define
	\begin{equation}
		b_c = 
		\begin{cases}
			\mu(e,c) &\text{ if $c \in L'(e)$, and}
			\\0 & \text{ otherwise.}
		\end{cases}
	\end{equation}
	Since $\mu(e,c) \leq 1$, it follows that 
	$$\sum_{c \in L(e)} b_c^2 \leq \sum_{ c \in L(e)} b_c= 1 \cdot X(x).$$
	Let $y$ be another outcome of the trials, which differs from $x$ in colours $J \subset L(e)$.
		Then $X(y) \geq X(x) - \sum_{c \in J} \mu(e,c)$.
		It follows that $X$ is $1$-certifiable with respect to $\Om^* = \es$.
	
	We apply Theorem~\ref{thm:talagrand-V} with $\Om^* = \es$, $b=1$, and $t = \N^{2/3}$.
	Note that $X \leq \om{L(e)} = \L < 3\e{\N}$ and thus $\Exp(X) \leq 3\e{\N}$.
	As $\L \geq \N$, it follows that $t > 96 \sqrt{\Exp (X)} +128$ for $\N$ large enough.
	Thus
	\begin{align}\label{equ:concentration-L}
		\Pr\left({\left|\om{L'(e)}-\L\cdot \K^k\right| > \N^{2/3}}\right) 
		&\leq \Pr\left({|X-\Exp(X)| > t } \right) \nonumber
		\\&\leq 4\exp\left({-\frac{\N^{4/3}}{8 (12\e{\N} + \N^{2/3})}}\right)\nonumber
		\\& \leq \exp\left({-\N^{1/10}}\right).
	\end{align}
	
	\newcommand{\TMC}{\textbf{TooManyColours}}
	\newcommand{\TMB}{\textbf{BlocksTooMany}}
	Now for part~\ref{itm:concentration-enemy-list} of Claim~\ref{cla:concentration}.
	Fix an edge $e$, a vertex $v \in e$ and a colour $c \in L(e)$. We first bound the probability, that an edge is assigned too many colours.
	Let $\TMC$ be the event that there exists $(f,c') \in N(e,v,c)$ such that the edge $f$ has been assigned more than $\N^{1/20}$ colours in step~\ref{prod:assign} of the procedure.
	The next claim bounds the probability of this event.
	We will use the following fact in its proof.
	Recall that $\mu\colon E\times \NATS \to [\exp\left(-\N^{1/20}\right),1]$.
	It follows that for every set of pairs $A \subset E \times \NATS$, 
	\begin{equation}\label{equ:weighted-vs-nonweighted}
		|A| \leq \exp\left(\N^{1/20}\right) \om{A}.
	\end{equation}

	\begin{claim}\label{cla:TooManyColours}
		$\Pr(\TMC)\leq \exp\left({-\N^{1/20}}\right)$.
	\end{claim}
	\begin{proof}[Proof of the claim]
		Let $A^{}$ denote the set of pairs $(f,c') \in N(e,v,c)$ for which $f$ has been assigned more than $\N^{1/20}$ colours in step~\ref{prod:assign} of the procedure.
		{Thus we can bound the expected value of $|{A^{}}|$ by
			\begin{align*}
				\Exp\left(|{A^{}|}\right) &\leq \sum_{(f,c') \in  {N(e,v,c)}} \sum_{S \in \binom{L(f)}{\N^{1/20}}} \prod_{c'' \in S} \mu(f,c'') \frac{1}{\L \ln \N}
				\\&\leq  \sum_{(f,c') \in  {N(e,v,c)}} \left(\frac{\e  }{\N^{1/20} \ln \N}\right)^{\N^{1/20}}
				\\&=  |N(e,v,c)| \cdot  \left(\frac{\e  }{\N^{1/20} \ln \N}\right)^{\N^{1/20}},
			\end{align*}
			where we used Lemma~\ref{lem:weighted-binomial-coeff-bound} with $N = |L(f)|$, $k=n^{1/20}$, and $p = \frac{1}{\ln n} \ge \frac{n}{\ell \ln n} \ge \sum_{c''\in L(f)} \mu(f,c'') \frac{1}{\ell \ln n}$.
			This allows us to estimate
			\begin{align*}
				\Exp\left(|{A^{}|}\right) &\leq 
				\exp\left(n^{1/20}\right) \cdot \om{N(e,v,c)} \cdot  \left(\frac{\e  }{\N^{1/20} \ln \N}\right)^{\N^{1/20}}
				\\&\le  
				\exp\left(n^{1/20}\right) \cdot \N \cdot  \left(\frac{\e  }{\N^{1/20} \ln \N}\right)^{\N^{1/20}}
				\\&\leq \exp\left({- \N^{1/20}}\right),
			\end{align*}
			where we used~\eqref{equ:weighted-vs-nonweighted} in the first line, $\om{N(e,v,c)} \le \N$ in the second line, and finally in the last line, we used that $n$ is large enough.
			The claim now follows by Markov's inequality.}
	\end{proof}

	Next we bound the probability, that a colour assignment blocks to many other colours.
	Let $\TMB$ be the event that for any $(f,c') \in N(e,v,c)$, $\alpha \in L(f)$, $w \in f$, and $(g, \beta) \in  N(f,w,\alpha)$,
	there are at least $\N^{1/20}$ pairs $(g', \beta') \in \bigcup_{w' \in g} N(g,w',\beta)$ such that $g'$ is assigned $\beta'$ in step~\ref{prod:assign} of the procedure.
	\begin{claim}\label{cla:TooManyBlocked}
		$\Pr(\TMB)\leq \exp\left({-\N^{1/20}}\right)$.
	\end{claim}
	\begin{proof}[Proof of the claim]
		Let $A^{}$ denote the set of pairs $(g, \beta)$ such that there exists $(f,c') \in N(e,v,c)$, $\alpha \in L(f)$, and $w \in f$ such that $(g, \beta) \in  N(f,w,\alpha)$ and at least $\N^{1/20}$ pairs $(g', \beta') \in \bigcup_{w' \in g} N(g,w',\beta)$ such that $g'$ is assigned $\beta'$ in step~\ref{prod:assign} of the procedure.  
		{We can thus bound the expected value of $|{A^{}}|$ by
			\begin{align*}
				\Exp\left(|A^{}|\right) & \leq \sum_{(f,c') \in  {N(e,v,c)}} 
				\sum_{ \alpha \in L(f)}
				\sum_{w \in f}
				\sum_{(g, \beta) \in  N(f,w,\alpha)}
				\sum_{S \in \binom{\bigcup_{w' \in g} N(g,w',\beta)}{\N^{1/20}}} 
				\prod_{(g',\beta') \in S} \frac{\mu(g',\beta')}{\L} \frac{1}{ \ln \N}
				\\&\leq  \sum_{(f,c') \in  {N(e,v,c)}} 
				\sum_{ \alpha \in L(f)}
				\sum_{w \in f}
				\sum_{(g, \beta) \in  N(f,w,\alpha)} \left(\frac{\e k  }{\N^{1/20} \ln \N}\right)^{\N^{1/20}},
			\end{align*} 
			where we used in the second line, we used Lemma~\eqref{lem:weighted-binomial-coeff-bound} with $N = |\bigcup_{w' \in g} N(g,w',\beta)|$, $k=n^{1/20}$ and $p = \frac{k}{\ln n} \geq \frac{kn}{\ell \ln n} \ge \sum_{\beta'\in \bigcup_{w' \in g} N(g,w',\beta)} \mu(g',\beta') \frac{1}{\ell \ln n}$.
			We continue to estimate
			\begin{align*}
				\Exp\left(|A^{}|\right) &\leq  \sum_{(f,c') \in  {N(e,v,c)}} 
				\sum_{ \alpha \in L(f)} k \cdot \exp\left(n^{1/20}\right) \cdot \N \cdot
				\left(\frac{\e k  }{\N^{1/20} \ln \N}\right)^{\N^{1/20}}
				\\&\leq  \sum_{(f,c') \in  {N(e,v,c)}} 
				\exp\left(n^{1/20}\right) \cdot \ell \cdot k \cdot \exp\left(n^{1/20}\right) \cdot \N \cdot
				\left(\frac{\e k  }{\N^{1/20} \ln \N}\right)^{\N^{1/20}}
			\end{align*} 
			where in the first line, we used~\eqref{equ:weighted-vs-nonweighted} to show that
			\[
			\sum_{(g, \beta) \in  N(f,w,\alpha)} 1 = |N(f,w,\alpha)| \le \exp\left(n^{1/20}\right) \om{N(f,w,\alpha)} \le \exp\left(n^{1/20}\right) \cdot n
			\]
			while $|f|=k$, and similarly in the second line, we used~\eqref{equ:weighted-vs-nonweighted} to show that $\sum_{ \alpha \in L(f)} 1 = |L(f)| \le \exp\left(n^{1/20}\right) \om{L(f)} \le \exp\left(n^{1/20}\right) \cdot \ell$.
			Finally, this gives
			\begin{align*}
				\Exp\left(|A^{}|\right) &\leq \exp\left(n^{1/20}\right) \cdot \N \cdot \exp\left(n^{1/20}\right) \cdot \ell \cdot  k \cdot n \exp\left(n^{1/20}\right) \cdot \left(\frac{\e k  }{\N^{1/20} \ln \N}\right)^{\N^{1/20}}
				\\&\leq \exp\left({ - \N^{1/20}}\right),
			\end{align*} 
			where (similarly to before) in the first line, we used~\eqref{equ:weighted-vs-nonweighted} to show that $\sum_{(f,c')\in N(e,v,c)} 1 = |N(e,v,c)| \le \exp\left(n^{1/20}\right) \om{N(e,v,c)} \le \exp\left(n^{1/20}\right) \cdot \N$, and for the last line we we used that $n$ is large enough.
			Now the claim follows by Markov's inequality.}
	\end{proof}

	To show that the random variable $\om{N''(e,v,c)}$ is highly concentrated, we express it as the difference of the following two random variables.
	\begin{itemize}
		\item Let $X_1$ contain the pairs $(f,c') \in N(e,v,c)$ such that all colours assigned to $f$ are removed from $f$.
		\item Let $X_2$ contain the pairs $(f,c') \in N(e,v,c)$ such that all colours assigned to $f$ are removed from $f$ {and} $L'(f)$ loses $c'$ at $w$ for some $w \in f \sm \{v\}$.
	\end{itemize}
	
	Observe that,  $ \om{N''(e,v,c)} = \om{X_1} -\om{X_2}$.
	We will use  Theorem~\ref{thm:talagrand-V} to show that $\om{X_1}$ and $\om{X_2}$ are each concentrated around their expected value.
	The analysis of $\om{X_1}$ and $\om{X_2}$ is almost the same.
	We will carry out the details for $\om{X_2}$.
	\newcommand{\Cert}{\text{Cert}}
	
	Let $({T}_i)$ denote the trials corresponding to the colour assignments and coin flips, respectively.
	{Let $x$ be an outcome of these trials and suppose that $x \notin \TMC \cup \TMB$.}
	For each $(f,c') \in N(e,v,c)$, let $C(f)$ be the set containing all colours assigned to the edge $f$ under the outcome of trials $x$.
	Suppose that $(f,c') \in X_2(x)$.
	Thus each colour $ \alpha \in C(f) $ must have been removed from $L'(f)$ under the outcome of trials $x$.
	In  particular, for every colour $\alpha \in C(f) $, there is a colour assignment or coin flip $T_{i_{f,\alpha}}(x)$ that witnesses $\alpha$ being removed from $f$.
	Similarly there is a colour assignment or coin flip $T_{i_{f,c'}}(x)$ that  {witnesses $c'$ being removed from $L'(f)$.} 
	Importantly, the witnessing trial $T_{i_{f,c'}}(x)$ corresponds either to an assignment of a colour to an edge $g\sim f$ with $v \notin g$ or to a coin flip $F(f,w,c')$ for $w \neq v$.
	For each $i$, define 
	$$W_i = \{(f,c',\alpha)\colon \text{ $(f,c')\in X_2(x)$ and $\alpha \in C(f) $ with $i = i_{f,\alpha}$}\}.$$
	and 
	$$b_i = \sum_{(f,c',\alpha) \in W_i}  \mu(f,c').$$
	Note that any $|W_i|\leq 1$ if $T_i$ corresponds to a coin flip.
	Since $G$ is $k$-uniform and $x \notin \TMB$, we have $|W_i| \leq k\N^{1/20}$ if $T_i$ corresponds to a colour assignment.
	Since $\mu(f,c') \leq 1$, we have $\mu(f,c')^2 \leq \mu(f,c')$.
	Moreover, as $(f,c') \notin \TMC$, it follows that $|C(f)| \leq \N^{1/10}$.
	Hence
	\begin{align*}
		\sum_i b_i^2 
		&\leq \left(k\N^{1/20}\right)^2 \sum_{i} \sum_{(f,c',\alpha) \in W_i} \mu(f,c')  
		\\&= k^2\N^{1/10} \sum_{\substack{(f,c') \in X_2(x)}} \sum_{\alpha \in C(f)} \mu(f,c')  
		\\&\leq k^2 \N^{2/10}  \sum_{(f,c') \in X_2(x) }   \mu(f,c')
		\\&= k^2 \N^{1/5} \cdot \om{X_2(x)}.
	\end{align*}
	Next, consider another outcome $y$ of the trials.
	If a pair $(f,c') \in X_2(x)$ is not in $X_2(y)$, then in particular $y_i \neq x_i$ for the witnessing trial $i = i_{f,\alpha}$ of some colour $\alpha \in C(f) $.
	Moreover, the removal of $f$ from $X_2(y)$ reduces $\om{X_2(y)}$ by $\mu(f,c_f)$ in comparison to $\om{X_2(x)}$.
	It follows that $\om{X_2(y)} \geq \om{X_2(x)} - \sum_{x_i \neq y_i} b_i$.
	Thus $\om{X_2(y)}$ is $b$-certifiable for $b=k^2\N^{1/5}$.

	For $j =1,2$, we apply Theorem~\ref{thm:talagrand-V} to $X_j$ with $\Om^* = \TMC \cup \TMB$, $b=k^2\N^{1/5}$, and  {$t = \N^{2/3}/2$}. 
	Note that ${\om{X_j} \leq \N}$ and thus $\sup(X) \leq \N$ as well as $\Exp(X_j) \leq \N$.
	Moreover, by Claim~\ref{cla:TooManyColours} and~\ref{cla:TooManyBlocked}, $\Pr{(\Om^*)} \leq \exp \left(n^{-1/10}\right)$.
	Using $\L \geq \N$, it follows that 
	$${t} = {\N^{2/3}/2} > 96 \sqrt{k^2\N^{1/5} \N} + 128 k^2\N^{1/5}    +8 \N \exp \left(n^{-1/10}\right) \geq 96  \sqrt{b\Exp (X)} +128 {b} + 8 \sup(X) \Pr(\Om^*)$$ 
	for $\N$ large enough.
	Thus
	\begin{align*}
		\Pr\left(\left| \om{N''(e,v,c)}- \Exp(\om{N''(e,v,c)}) \right| > \N^{2/3} \right) 
		&\le \Pr \left(\left| \Exp(\om{X_1})- \om{X_1} \right| > \frac{1}{2} \cdot \N^{2/3} \right) 
		\\&\quad + \Pr \left(\left| \Exp(\om{X_2})- \om{X_2} \right| > \frac{1}{2} \cdot \N^{2/3} \right) 
		\\ &\leq 8\cdot \exp\left({-\frac{\N^{4/3}}{8k^2\N^{1/5} (4\N+\N^{2/3})}}\right) + 8 \cdot \exp\left({\N^{-1/10}}\right)
		\\&\leq  16 \cdot {\exp\left({-\N^{1/10}}\right)}. 
	\end{align*}
	By Claim~\ref{cla:expectation}, we have {$\Exp \left({\om{N''(e,v,c)}}\right) \leq \N\cdot \K^{k-1} \cdot \left(1-\frac{{1-\eps/2}}{\ln \N} \K^{k} \right) \le n'-n^{2/3}$.} This shows part~\ref{itm:concentration-enemy-list} of Claim~\ref{cla:concentration}.

	\section{Finishing blow} \label{sec:finisher}
	In this section, we prove Lemma~\ref{lem:edge-finisher}.
	We will actually show a more general result for vertex colourings.
	
	Let $G=(V,E)$ be a $k$-uniform linear hypergraph.
	A \emph{vertex correspondence} $\sigma$ of $G$ consists of integer permutations $\sigma_{v,u}=\sigma_{u,v}$ for all adjacent vertices $v,u$.
	We say that $(u,c)$ \emph{blocks} $(v,c')$ if $\sigma_{u,v}(c) = c'$.
	An assignment $(L,\mu)$ of \emph{weighted} lists of colours to the vertices of $G$ consists of a lists of colours $L(v)$ and weight functions $\mu(v)\colon L(v) \to [0,1]$.
	An $(L,\sigma)$-colouring is a function $\gamma \colon V \to \NATS$ such that
	\begin{itemize}
		\item $\gamma(c) \in L(v)$ for every $v \in V$, and
		\item $(v,\gamma(v))$ does not block $(u,\gamma(u))$ for all adjacent vertices $u,v$.
	\end{itemize}
	The set $N_{G,L,\sigma}(v,c)$ contains all pairs $(w,c') \in V \times \NATS$ such that $w \in N(v)$, $c' \in L(w)$ and $(w,c')$ blocks $(e,c)$.
	\begin{lemma}[Finisher vertex version]\label{lem:vertex-finisher}
		Let $G =(V,E)$ be a $k$-uniform linear hypergraph with an vertex correspondence $\sigma$ and weighted lists of colours $(L,\mu)$ assigned to its vertices.
		Suppose that each list has finite size.
		Let $\L,\N >0$ such that, for all $v \in E$, all of the following hold:
		\begin{enumerate}[\upshape (i)]
			\item \label{itm:ratio} $\L/\N\geq 3\e $,
			\item \label{itm:sum-over-colours} $\om{L(v)} \geq \L$ for all $v \in V$,
			\item \label{itm:sum-over-neighbours} $\om{N_{G,L,\sigma}(v,c)} \leq \N$ for all $v \in V$ and $c \in L(v)$.
		\end{enumerate}
		Then $G$ has an $(L,\sigma)$-colouring.
	\end{lemma}
	Note that we can derive Lemma~\ref{lem:edge-finisher} by applying Lemma~\ref{lem:vertex-finisher} to the {link graph}
	
	\begin{proof}[Proof of Lemma~\ref{lem:vertex-finisher}]
		For each vertex $v$, select independently a colour $c \in \NATS$ according to the distribution $\mu(v,c)/\L$.
		For every edge $uv$ and colours $c,c' \in \NATS$ such that $(u,c)$ blocks $(v,c')$, we let $B({u,v,c,c'})$ be the (bad) event that $u$ is assigned $c$ and $v$ is assigned $c'$.
		Note that $B({u,v,c,c'})$ is independent of all events whose vertices are not adjacent to $u$ or $v$.
		Moreover, $B({u,v,c,c'})$ has probability $\mu(u,c)\cdot \mu(v,c')/\L^2$.
		Also note that, as each list has finite size, there are only finitely many events $B(u,v,c,c')$.
		{To apply Theorem~\ref{thm:local-lemma}, we define  $x_{u,v,c,c'}$ by setting
			\begin{align*}	 
				x_{u,v,c,c'} =   1-\frac{1}{\frac{\mu(u,c)\cdot \mu(v,c')}{2\L \N}+1}.
			\end{align*}
			Since $  1- \frac{1}{1/x+1} = \frac{1}{x+1}$ for every $x \neq -1$, we may bound
			\begin{align}	\label{equ:local-lemma-bound-x}
				x_{u,v,c,c'}&=  \frac{1}{\frac{2\L \N}{\mu(u,c)\cdot \mu(v,c')}+1} \nonumber
				\\ &\geq	\frac{\mu(u,c)\cdot \mu(v,c')}{3\L \N } \nonumber
				\\ & 	\overset{\text{\ref{itm:ratio}}}{\geq}  \e \cdot  \frac{\mu(u,c)\cdot \mu(v,c')}{\L^2}		 = \e \cdot \Pr\left({B(u,v,c,c')}\right).
		\end{align}}

		On the other hand,  using that $1-x \geq \e^{-{\frac{x}{1-x}}} =\e^{{1-\frac{1}{1-x}}}$, we obtain
		\begin{align}\label{equ:local-lemma-dependency}
			&\quad 
			\prod_{\substack{uw \in E}} \prod_{ \substack{b \in L(u),~b' \in L(w), \\ \text{$(u,b)$ blocks $(w,b')$} } } (1-x_{u,w,b,b'}) \cdot 
			\prod_{\substack{vw \in E}} \prod_{\substack{b \in L(v),~b' \in L(w), \\\text{$(v,b)$ blocks $(w,b')$} } }  (1-x_{v,w,b,b'})\nonumber
			\\ &\geq   
			\exp\left({-\sum_{uw \in E} \sum_{ \substack{b \in L(u),~b' \in L(w), \\ \text{$(u,b)$ blocks $(w,b')$} } } \frac{\mu(u,b)\cdot \mu(w,b')}{2\L  \N}}\right) \cdot 
			\exp\left({-\sum_{vw \in E} \sum_{ \substack{b \in L(v),~b' \in L(w), \\ \text{$(v,b)$ blocks $(w,b')$} } } \frac{\mu(v,b)\cdot \mu(w,b')}{2\L \N}}\right)\nonumber
			\\ &=    
			\exp\left(-\sum_{ b \in L(u)}{\frac{\mu(u,b)}{2\L \N}  \sum_{ (w,b') \in N_{G,L,\sigma}(u,b)} \mu(w,b')}\right) \cdot 
			\exp\left(-\sum_{ b \in L(v)}{\frac{\mu(v,b)}{2\L \N}  \sum_{ (w,b') \in N_{G,L,\sigma}(v,b)} \mu(w,b')}\right) \nonumber  
			\\ &\overset{\text{\ref{itm:sum-over-neighbours}}}{\geq}   
			\exp\left(-\sum_{ b \in L(u) }{\frac{\mu(u,b)}{2\L}}\right) \cdot 
			\exp\left(-\sum_{ b \in L(v)}{\frac{\mu(v,b)}{2\L}}\right)  
			\overset{\text{\ref{itm:sum-over-colours}}}{\geq}  \frac{1}{\e}.
		\end{align}
		In combination~\eqref{equ:local-lemma-bound-x} and~\eqref{equ:local-lemma-dependency}, show that the conditions of Theorem~\ref{thm:local-lemma} are satisfied.
		Thus with positive probability none of the events $B(u,v,c,c')$ occur.
		It follows that $G$ has an $(L,\sigma)$-colouring.
	\end{proof}
	
	\section{Open problems}\label{sec:open-problems}
	In light of our results, the most basic open question is whether one can prove Theorems~\ref{thm:main-simple}--\ref{thm:main} without the polylogarithmic relation between the minimum and maximum degree.
	While in some situations of local vertex colourings such a relation is in fact necessary~\cite{DJK20}, we believe that in the setting of list edge colouring it might as well be redundant.
	In our proof, the necessity for this relation arises from the use of the Lov\'asz Local Lemma (see Claim~\ref{cla:expectation}), which is an integral part of the approach.
	Thus, it seems that new ideas would be required to circumvent this condition.
	
	Another interesting problem consists in improving the error terms of our bounds. 
	Molloy and Reed~\cite{MR00} improved Theorem~\ref{thm:kahn} to $\chi'_{\ell}(G) \leq \Delta(G) + O(\sqrt{\Delta} \log^4 \Delta)$.
	It is therefore natural to ask whether similar results could be obtained in the local setting.
	One obstacle with regards to this is that, in the situation of correspondence colourings, it does not seem to be possible to reserve a set of colours as required for the proof of Molloy and Reed~\cite{MR00}.
	
	Finally, it might as well be true that the List Colouring Conjecture (Conjecture~\ref{con:list-colouring-conjecture}) extends to the local setting.
	However, since the original conjecture is already quite an illusive problem, let us instead finish with a less daunting question.
	Recall that Vizing's theorem states that every graph $G$ of maximum degree $\Delta$ has an $L$-edge-colouring from the lists $L(e)=\{1,\dots,\Delta+1\}$.
	It is our belief that this result can be strengthened to the local setting.
	\begin{conjecture}[Local Vizing's theorem]\label{con:local-vizing}
		Let $G$ be a graph with a list assignment $L$ of $E(G)$ such that for every edge $e=uv$
		$$ L(e)= \max \{1,\dots,\max \{\deg(u),\deg(v)\}+1\}.$$
		Then there is an $L$-edge-colouring of $G$.
	\end{conjecture}
	
	{We remark that Conjecture~\ref{con:local-vizing} was recently confirmed by Christiansen~\cite{Chr22} using a clever induction argument.}

	\section{Acknowledgments}
	The authors would like to thank Louis Esperet and Jan Van Den Heuvel for stimulating discussions on the matching polytope.
	{We also thank two anonymous referees for many helpful remarks.}
	
	\bibliographystyle{amsplain}
	\bibliography{local-kahn}
	
	\appendix
	\section{Appendix}
	\label{sec:appendix}

	In this appendix, we prove Theorem~\ref{thm:talagrand-V}. First we need the following lemma.
	
	\begin{lemma}\label{lem:CertScaling}
		If $X$ is $b$-certifiable, then for every $s \ge 0$ and $\omega\in\Omega\setminus\Omega^*$ such that $X(\omega) \geq s$, there exists a $b$-certificate for $X, \omega, s$, and $\Omega^*$.
	\end{lemma}
	\begin{proof}
		Since $X$ is $b$-certifiable, there exists a $b$-certificate for $X,\omega, s':=X(\omega)$ and $\Omega^*$ by definition, that is there exists $I\subseteq \{1,\ldots,n\}$ and a vector $(c_i'\colon i\in I)$ with $\sum_{i\in I} (c_{i}')^2 \le bs'$ such that for all $I'\subseteq I$, we have that
		\begin{equation*}
			X(\omega') \geq s' - \sum_{i\in I'} c_i',
		\end{equation*}
		for all $\omega' = (\omega'_1, \dots, \omega'_n)\in\Omega\setminus\Omega^*$  such that $\omega_i=\omega_i'$ for all $i\in I\setminus I'$.
		
		Note that $s'=X(\omega)\ge s$ by assumption.
		{If $s=0$, then the empty certificate is a $b$-certificate for $X,\omega,0$ and $\Omega^*$ as desired since $X\ge 0$. So we assume that $s > 0$ and hence $s' > 0$.}
		Now for all $i\in I$, we define $c_i := c_i' \cdot \frac{s}{s'}$. We claim that $I$ and $(c_i\colon i\in I)$ is a $b$-certificate for $X,\omega,s$ and $\Omega^*$ as desired. To see this, note that
		$$\sum_{i\in I} c_i^2 = \left(\frac{s'}{s}\right)^2 \sum_{i\in I} (c_i')^2 \le \left(\frac{s}{s'}\right)^2 bs' \le bs,$$
		where the last inequality follows since $s\le s'$. Furthermore, for all $\omega' = (\omega'_1, \dots, \omega'_n)\in\Omega\setminus\Omega^*$  such that $\omega_i=\omega_i'$ for all $i\in I\setminus I'$, we have that
		$$X(\omega') \ge s' - \sum_{i\in I'} c_i' = \frac{s'}{s} \left( s- \sum_{i\in I'} c_i \right) \ge s- \sum_{i\in I'} c_i,$$
		where the last inequality follows since $s'\ge s$. This completes the claim and hence the proof of the lemma.
	\end{proof}
	
	In order to prove Theorem~\ref{thm:talagrand-V}, we show the following theorem which yields concentration around the median under the same conditions.
	\begin{theorem}\label{exceptional talagrand's with median}
		If $X$ is $b$-certifiable with respect to $\Omega^*$, then for any $t > 0$,
		\begin{equation*}
			\Prob{|X - \Med(X)| > t} \leq 4\exp\left({-\frac{t^2}{4b(\Med(X) + t)}}\right) + 4\Prob{\Omega^*}
		\end{equation*}
	\end{theorem}
	
	We then prove that the expectation and median are close as in the following lemma.
	\begin{lemma}\label{expectation close to median}
		If $X$ is $b$-certifiable with respect to $\Omega^*$ and $M = \sup X$, then
		\begin{equation*}
			|\Expect{X} - \Med(X)| \leq 48\sqrt{b\Expect{X}} +  64b + 4M\Prob{\Omega^*}.
		\end{equation*}
	\end{lemma}
	\begin{proof}
		Let $Y = X + \Expect{X}$.
		Note that $\Expect{Y} - \Med(Y) = \Expect{X} - \Med(X)$, $\Med(Y) \geq \Expect{X} > 0$, and $\Expect{Y} \leq 2\Expect{X}$.
		(Since $X$ is non-negative, the case $\Expect{X}=0$ is trivial.)
		Note also that
		\begin{equation*}
			|\Expect{Y} - \Med(Y)| \leq \Expect{|Y - \Med(Y)|}.
		\end{equation*}
		
		Let $L = \lfloor M/(\sqrt{b\Med(Y)})\rfloor$, and note that $|Y - \Med(Y)| \leq  (L + 1)\sqrt{b\Med(Y)}$.  By partitioning the possible values of $|Y - \Med(Y)|$ into intervals of length $\sqrt{b\Med(Y)}$, we get
		
		\begin{align*}
			\Expect{|Y - \Med(Y)|} &\leq \begin{aligned}[t]
				\sum_{\ell=0}^L \sqrt{b\Med(Y)}(\ell + 1) &\left(\Prob{|Y - \Med(Y)| \geq \ell \sqrt{b\Med(Y)}}\right.\\
				&\left. - \Prob{|Y - \Med(Y)| \geq (\ell + 1) \sqrt{b\Med(Y)}}\right)\end{aligned}\\
			&=  \sqrt{b\Med(Y)} \sum_{\ell=0}^L  \Prob{|Y - \Med(Y)| \geq \ell \sqrt{b\Med(Y)}} .
		\end{align*}
		By applying Theorem~\ref{exceptional talagrand's with median} with $t=\ell \sqrt{b\Med(Y)}$ to every summand,
		\begin{equation*}
			\Expect{|Y - \Med(Y)|} \leq 4\sqrt{b\Med(Y)}\sum_{\ell=0}^L \left(\exp\left({-\frac{\ell^2b\Med(Y)}{4b(\Med(Y) + \ell \sqrt{b\Med(Y)})}}\right) + \Prob{\Omega^*}\right).
		\end{equation*}
		Note that for each $\ell\in\{0, \dots, L\}$,
		\begin{multline*}
			\exp\left({\frac{\ell^2b\Med(Y)}{4b(\Med(Y) + \ell \sqrt{b\Med(Y)})}}\right) \leq \exp\left({\frac{\ell^2b\Med(Y)}{8b\max\{\Med(Y), \ell \sqrt{b\Med(Y)}\}}}\right)\\
			\leq \exp\left({\frac{\ell^2b\Med(Y)}{8b\Med(Y)}}\right) + \exp\left({\frac{\ell^2b\Med(Y)}{8b\ell\sqrt{b\Med(Y)}\}}}\right) = \exp\left(\ell^2/8\right) + \exp\left(\frac{\ell\sqrt{\Med(Y)}}{8\sqrt{b}}\right).
		\end{multline*}
		Recall that $L = \lfloor M/(\sqrt{b\Med(Y)})\rfloor$, and hence 
		\begin{equation*}
			{4\sqrt{b\Med(Y)}\sum_{\ell=0}^L\Prob{\Omega^*} \leq 4M\Prob{\Omega^*}.}
		\end{equation*} 
		Therefore
		\begin{equation*}
			\Expect{|Y - \Med(Y)|} \leq 4\sqrt{b\Med(Y)}\sum_{\ell=0}^\infty \left(\exp\left({-\ell^2/8}\right) + \exp\left({-\frac{\ell\sqrt{\Med(Y)}}{8\sqrt{b}}}\right)\right) + 4M\Prob{\Omega^*}.
		\end{equation*}
		Note that $\sum_{\ell=0}^\infty e^{-\ell x} = \frac{1}{1 - e^{-x}}$.  Note also that $\frac{x}{2} \leq 1 - e^{-x}$ if $x < \frac{3}{2}$.   Since $\frac{1}{1 - e^{-x}} < 2$ when $x\geq \frac{3}{2}$, $\frac{1}{1 - e^{-x}} \leq \max\{2, \frac{2}{x}\}$.  Therefore
		\begin{equation*}
			\sum_{\ell=0}^\infty \exp\left(-\frac{\ell\sqrt{\Med(Y)}}{8\sqrt{b}}\right) \leq \max\left\{2, \frac{16\sqrt{b}}{\sqrt{\Med(Y)}}\right\}.
		\end{equation*}
		
		Note that $\sum_{\ell=0}^\infty e^{-\ell^2/ 8} < 4$.  Therefore
		\begin{equation*}
			\Expect{|Y - \Med(Y)|} \leq 4\sqrt{b\Med(Y)}\left(4 + \max\left\{2, \frac{16\sqrt{b}}{\sqrt{\Med(Y)}}\right\}\right) + 4M\Prob{\Omega^*}.  
		\end{equation*}
		Since the maximum of two numbers is at most their sum,
		\begin{equation*}
			\Expect{|Y - \Med(Y)|} \leq 24\sqrt{b\Med(Y)} + 64b + 4M\Prob{\Omega^*}.
		\end{equation*}
		Since {$\Med(Y) \leq 2\Expect{Y} = 4\Expect{X}$,}
		\begin{equation*}
			\Expect{Y - \Med(Y)|} \leq 48\sqrt{b\Expect{X}} +  64b + 4M\Prob{\Omega^*},
		\end{equation*}
		as desired.
	\end{proof}

	Now we can prove Theorem~\ref{thm:talagrand-V} assuming Theorem~\ref{exceptional talagrand's with median}.
	\begin{proof}[Proof of Theorem \ref{thm:talagrand-V}]
		Since $t > 96\sqrt{b\Expect{X}} +  128b + 8M\Prob{\Omega^*},$
		\begin{equation}\label{t over 2 bound}
			\frac{t}{2} > 48\sqrt{b\Expect{X}} +  64b + 4M\Prob{\Omega^*}.
		\end{equation}
		By applying Lemma \ref{expectation close to median} and then \eqref{t over 2 bound},
		\begin{equation*}
			\Prob{|X - \Expect{X}| > t} \leq \Prob{|X - \Med(X)| > \frac{t}{2}}.
		\end{equation*}
		Since $\Med(X) \leq 2\Expect{X}$, Theorem~\ref{exceptional talagrand's with median} implies that
		\begin{align*}
			\Prob{|X - \Med(X)| > \frac{t}{2}} &\leq 4\exp\left({-\frac{(t/2)^2}{4b(2\Expect{X} + (t/2))}}\right) + 4\Prob{\Omega^*},\\
			&=4\exp\left({-\frac{t^2}{8b(4\Expect{X} + t)}}\right) + 4\Prob{\Omega^*}
		\end{align*}
		as desired.
	\end{proof}

	It remains to prove Theorem~\ref{exceptional talagrand's with median}.
	Let $((\Omega_i, \Sigma_i, \mathbb P_i))_{i=1}^n$ be probability spaces and $(\Omega, \Sigma, \mathbb P)$ their product space.  For a set $A\subseteq \Omega$ and event $\omega\in\Omega$, let
	\begin{equation}
		d(\omega, A) = \sup_{||\alpha||=1}\left\{\tau \colon \sum_{i\colon\omega_i\neq\omega'_i}\alpha_i\geq\tau \rm{\ for\ all\ }\omega'\in A\right\}.
	\end{equation}
	
	We use the original version of Talagrand's Inequality.
	\begin{theorem}[Talagrand's Inequality \cite{Tal95}]\label{OG Tala}
		If $A,B\subseteq \Omega$ are measurable sets such that for all $\omega\in B$, $d(\omega, A) \geq \tau$, then
		$$\Prob{A}\Prob{B}\leq e^{\frac{-\tau^2}{4}}.$$
	\end{theorem}
	
	We can now prove Theorem~\ref{exceptional talagrand's with median}.
	\begin{proof}[Proof of Theorem~\ref{exceptional talagrand's with median}]
		It suffices to show that
		\begin{equation}\label{one sided estimation}
			\Prob{X\leq \Med(X) - t} \leq 2\exp\left({-\frac{t^2}{8b(\Med(X) + t)}}\right) + 2\Prob{\Omega^*}
		\end{equation}
		and
		\begin{equation}\label{other one sided estimation}
			\Prob{X\geq \Med(X) + t} \leq 2\exp\left({-\frac{t^2}{8b(\Med(X) + t)}}\right) + 2\Prob{\Omega^*}.
		\end{equation}
		Let 
		\begin{align*}
			&A = \{\omega\in \Omega\backslash \Omega^* \colon X(\omega) \geq \Med(X) + t\}, \text{ and}\\
			&B = \{\omega\in \Omega\backslash \Omega^* \colon X(\omega) \leq \Med(X) \}.
		\end{align*}
		
		We need to show the following.
		\begin{claim}\label{A and B satisfy OG Tala}
			For all $\omega\in A$, $d(\omega, B) \geq \frac{t}{\sqrt{b(\Med(X) + t)}}$.
		\end{claim}
		\begin{proofclaim}
			Since $X$ is $b$-certifiable, we have by Lemma~\ref{lem:CertScaling} that there exists a $b$-certificate for $X, \omega, \Med(X) + t,$ and $\Omega^*$, that is a subset $I\subseteq\{1,\ldots,n\}$ and a vector $(c_i\colon i\in I)$. 
			
			Let $\omega'\in B$.  Let $I' = \{i\in I\colon  \omega_i \ne \omega'_i\}$. By definition of $b$-certificates and since {$X(\omega') \leq \Med(X)$}, we have that 
			$${\sum_{i\in I'} c_i \ge (\Med(X)+t)-X(\omega') \ge t.}$$
				
			Let $r = \sqrt{\sum_{i\in I}c_i^2}$. By the definition of $b$-certificates, we have that $\sum_{i\in I}c_i^2 \le b(\Med(X)+t)$ and hence $r\le \sqrt{b(\Med(X)+t)}$.
			
			Now we set $\alpha = (\alpha_i\colon i\in[n])$ where $\alpha_i = c_i/ r$ if $i\in I$ and $\alpha_i = 0$ otherwise. Thus 
			$$\sum_{i\in I'} \frac{c_i}{r} \ge \frac{t}{r} \ge \frac{t}{\sqrt{b(\Med(X) + t)}}.$$
			Hence $d(\omega,B) \ge \frac{t}{\sqrt{b(\Med(X) + t)}}$ as desired.
		\end{proofclaim}
		
		Now \eqref{other one sided estimation} follows from Claim \ref{A and B satisfy OG Tala} and Theorem~\ref{OG Tala}.  The proof of~\eqref{one sided estimation} is similar, so we omit it.
	\end{proof}

\end{document}